\documentclass[UTF8,12pt]{article}
\usepackage[dvipsnames,svgnames,x11names]{xcolor}
\usepackage{upgreek}
\usepackage{listings}
\usepackage{amsmath,amssymb}
\usepackage{longtable}
\usepackage{enumerate}
\usepackage{amsthm}
\usepackage{multirow}
\usepackage{mathrsfs}
\usepackage{pgf,tikz}
\usepackage{pgfplots}
\usepackage{makecell}
\usepackage{graphicx}
\usepackage{subfigure}
\usepackage{caption}
\usepackage[ruled,vlined,linesnumbered]{algorithm2e}
\usepackage{bm}
\usepackage{bbm}
\usepackage{lscape}
\usetikzlibrary{arrows}
\usepackage{geometry}
\usepackage{booktabs}
\usepackage{url} 
\usepackage{ragged2e}
\usepackage{hyperref}
\usepackage{float}
\renewcommand\leq{\leqslant}
\renewcommand\geq{\geqslant}
\renewcommand\pi{\uppi}
\newcommand{\rmd}{\mathrm{d}}
\newcommand{\transpose}{^\top}
\newcommand{\trace}{\mathrm{Trace}}

\newtheorem{theorem}{Theorem}[section]
\newtheorem{prop}[theorem]{Proposition}
\newtheorem{rmk}[theorem]{Remark}
\newtheorem{defn}[theorem]{Definition}
\newtheorem{example}[theorem]{Example}

\DeclareMathOperator{\vol}{vol}

\usepackage{array}
\newcolumntype{L}[1]{>{\RaggedRight\arraybackslash}p{#1}} % left with length
\newcolumntype{C}[1]{>{\Centering\arraybackslash}p{#1}} % center with length
\newcolumntype{R}[1]{>{\RaggedLeft\arraybackslash}p{#1}} % right with length

\newcolumntype{N}[1]{>{\RaggedLeft\arraybackslash $ }p{#1}<{$}} % math right with length
\newcolumntype{M}{>{$}r<{$}} % math right
\newcolumntype{P}{>{$}l<{$}} % math left

\makeatletter
\@addtoreset{equation}{section}
\makeatother

\newcommand{\adda}[1]{{\color{black}{#1}}}
\newcommand{\addb}[1]{{\color{black}{#1}}}

\usepackage{ulem}

\begin{document}
\title{An ODE approach to multiple choice polynomial programming}
\author{Sihong Shao\footnotemark[1],
\and Yishan Wu\footnotemark[2]}

\renewcommand{\thefootnote}{\fnsymbol{footnote}}
\footnotetext[1]{CAPT, LMAM and School of Mathematical Sciences, Peking University, Beijing 100871, China. Email: \texttt{sihong@math.pku.edu.cn}}
\footnotetext[2]{CAPT, LMAM and School of Mathematical Sciences, Peking University, Beijing 100871, China. Email: \texttt{wuyishan@pku.edu.cn}}
\date{\today}
\maketitle

\begin{abstract}
%We present an ODE-based algorithm for multiple choice polynomial programming. It can be derived from approximating a continuous version of simulated annealing. Theoretical analysis also shows that the algorithm can find local optimum. We run numerical experiments on MAX-$k$-CUT and star discrepancy calculation. Compared to heuristic algorithms, our method can produce solutions with similar quality with much less computing resource. 

We propose an ODE approach to solving multiple choice polynomial programming (MCPP) after assuming that the optimum point can be approximated by the expected value of so-called thermal equilibrium as usually did in simulated annealing. The explicit form of the feasible region and the affine property of the objective function are both fully exploited in transforming the MCPP problem into the ODE system. We also show theoretically that a local optimum of the former can be obtained from an equilibrium point of the latter. Numerical experiments on two typical combinatorial problems, 
MAX-$k$-CUT and the calculation of star discrepancy, demonstrate the validity of the ODE approach, 
and the resulting approximate solutions are of comparable quality to those obtained by the state-of-the-art heuristic algorithms but with much less cost. 
When compared with the numerical results obtained by using Gurobi to solve MCPP directly,
our ODE approach is able to produce approximate solutions of better quality in most instances. 
This paper also serves as the first attempt to use a continuous algorithm for approximating the star discrepancy. 

%?????? Multiple choice polynomial programming represents a class of combinatorial optimization problems, where the variables are binaries and divided into groups that the sum of variables in each group must be $1$, and the objective function is polynomial. Several combinatorial optimization problems such as MAX-$k$-CUT can be stated in this form. There are various algorithms for each specific problem. We proposed a general algorithm framework based on ODE that can be applied in several problems and produce good and stable solutions. We explored the multilinearity of the objective function which leads to a key simplification in the forward equation of random variables in a continuous time Markorv chain. This then make it possible to form a simple, calculable ODE. We did numerical experiments on MAX-$k$-CUT and star discrepancy calculation. Compared to heuristic algorithms, our method can produce solutions with similar quality with much less computing resource. Additionally, this framework is flexible and several techniques in numerical ODE and continuous optimization are applicable for improving the algorithm. We set an example by introducing variable step and Nesterov's acceleration (Five sentences????).% 5句话左右 

\vspace*{4mm}

% 65K05, % Numerical mathematical programming methods 好像太宽泛了
% 68W25, % Approximation algorithms 也太宽泛了

\end{abstract}

\section{Introduction}
\label{sec:intro}
%\add{We should mention that it is the cardinality of each $I_j$ that determines the structure of the feasible region, rather than the particular devision.} 

We consider the following pseudo-Boolean optimization problem 
\begin{equation}\label{eq:mcpp0}
\begin{aligned}
\min_{x\in \{0,1\}^n} &&& f(x),\\
\text{s.t.} &&& \sum_{i\in I_j} x_i=1,\ j=1,2,\dots m,
\end{aligned}
\end{equation}
where $f$ is a polynomial function, $x$ is {an} $n$-dimensional Boolean vector, the indices $[n]: = \{1,2,\dots,n\}$ are divided into $m$ disjoint subsets $I_1,I_2,\dots,I_m$,
and the cardinality of each $I_j$, denoted by {$d_j := |I_j|$}, must be greater than $1$.  Then $x$ is accordingly divided into $m$ vectors $x^{(1)}, x^{(2)},\dots,x^{(m)}$
where each $x^{(j)}$ picks out all the entries in $I_j$ of $x$. The constraint{s} mean that, for each $j$,  there exists exact one element that equals to $1$ in the subvector $x^{(j)}$, implying that only one is determined from $d_j$ choices and thus $m$ items chosen from $n$ choices in total. A group of entries of $x$ in a single $I_j$ represents a decision out of finite choices. More precisely, we use multiple choice polynomial programming (MCPP) to call the problem \eqref{eq:mcpp0}, which is capable of dealing with various problems in diverse disciplines \cite{boros_pseudo-boolean_2002}, for example, the MAX-$k$-CUT \cite{ma_multiple_2017}, star discrepancy \cite{gnewuch_new_2011} problems and SAT \cite{kirkpatrick_optimization_1983}. It should be pointed out that studies on integer linear programming problems with multiple choice constraints,  termed  ``multiple choice programming" (MCP), can date back to \cite{healy_multiple_1964}. Since then several methods have been developed {(see, e.g., \cite{singh_multi-choice_2017} for a detailed review)}, but most of them are not designed for nonlinear objective functions. Although MCPP can be transformed into MCP by defining new variables to represent monomials, an exploration in this direction will not be presented here. 

As a typical $0$-$1$ programming problem, {MCPP} can also be treated with standard {mixed} integer {nonlinear} programming solvers, such as Gurobi~\cite{gurobi}, IBM-CPLEX~\cite{cplex}, and SCIP~\cite{BestuzhevaBesanconChenetal.2021},
but few of them are designed for general nonlinear programming. For example, to apply Gurobi, one has to reformulate the nonlinear terms in MCPP instances into linear and/or quadratic forms by defining new variables and new constraints, thereby greatly increasing the problem size. This defect becomes even severer in calculating the star discrepancy since the degree of corresponding objective function is nothing but the dimension of underlying space and is usually much greater than $1$ (see Eq.~\eqref{eq:MCPP_disc} in Section~\ref{sec:discrepancy}). % 星偏差问题对应的 order >2 问题更严重

Alternatively, developing continuous approaches {for} discrete problems has gotten more attention since, e.g., the Hopfield network \cite{hopfield_neurons_1984}, an early ODE approach, was proposed for the $0$-$1$ quadratic programming in 1980s. The Hopfield network has been extended to other combinatorial optimization problems, see, e.g. \cite{jagota_approximating_1995, wang_improved_2006}, where many useful mathematical techniques in dynamical system have been introduced. In mostly recent work \cite{niu_discrete_2022}, an ODE approach was 
proposed through a quartic penalty approximation of the Boolean polynomial program. 
In line with this, here we propose to use the solutions of the following ODE system 
\begin{equation}
\label{eq:main_ODE}
\left\{
\begin{array}{ll}
\displaystyle 
\frac{\rmd{y^{(j)}_i}}{\rmd t}=-y^{(j)}_i+\sigma_i(-\Phi^{(j)}(y);1/T),\quad j\in [m],\ i\in[d_j],\\
\displaystyle y(0)=y_0 \in [0,1]^n,
\end{array}
\right.
\end{equation}
to approximate the solutions of MCPP, where the time-dependent vector $y(t): [0,+\infty)\rightarrow \mathbb{R}^n$ is divided as $x$ in Eq.~\eqref{eq:mcpp0} does, the initial data $y_0$ is required to satisfy the ``continuous" multiple choice constraint: $\sum_{i=1}^{d_j} (y_0)_i^{(j)} = 1$ for all $j\in[m]$, $\Phi^{(j)}$ denotes the partial derivative of $f$ with respect to $x^{(j)}$: $\Phi^{(j)}_i = {\partial f}/{\partial x_i^{(j)}}$, $T$ is a positive parameter called ``temperature", {and} $\sigma_i(z; \beta): \mathbb{R}^d\times\mathbb{R}_+ \rightarrow (0,1)$ gives the softmax function defined by 
\begin{equation}\label{eq:soft}
\sigma_i(z;\beta)=\frac{\exp\left(\beta z_i\right)}{\sum_{k=1}^d\exp\left(\beta z_{k}\right)},
\end{equation}
with $d$ being the dimension of input argument $z$. Note in passing that $y^{(j)}_i(t)\in(0,1)$ for arbitrary $t>0$, i.e., $y(t): [0,+\infty)\rightarrow (0,1)^n$. We are able to prove that the equilibrium points of the ODE system \eqref{eq:main_ODE} represent the local optimum 
solutions of MCPP \eqref{eq:mcpp0}. Therefore, we just have to numerically integrate \eqref{eq:main_ODE} until we find an equilibrium (usually fast with fewer iterations) which can be rounded into a local solution of MCPP.  Along the way, various well-established techniques in numerically solving ODEs can be incorporated to improve the efficiency without sacrificing the solution quality.
Actually, the ODE system \eqref{eq:main_ODE} can be regarded as a continuous version of simulated annealing (SA) \cite{kirkpatrick_optimization_1983} by starting from a continuous-time Markov chain and exploiting two intrinsic properties of MCPP \eqref{eq:mcpp0}.
One is that the explicit form of the feasible region is so straightforward that the state space of the Markov chain can be easily determined. The other is that the objective function of MCPP, $f(x^{(1)},x^{(2)},\dots,x^{(m)})$,  is affine with respect to $x^{(j)}$ for all $j \in [m]$ (otherwise we can modify it equivalently),
with which the definition of transition rate can be significantly simplified.
It should be noted that Eq.~\eqref{eq:main_ODE} may recover the ODE system used in \cite{peterson_neural_1988} for the unconstrained binary quadratic programming (see Example~\ref{ex1}). 

We apply the proposed ODE approach into two typical NP-hard problems:
MAX-$k$-CUT with $k=2,3,4,5$ and the star discrepancy calculation,
and two state-of-the-art heuristic algorithms: the multiple operator heuristic (MOH) \cite{ma_multiple_2017} and the improved threshold accepting (TA\_improved) \cite{gnewuch_new_2011} methods are employed to as the reference \adda{of the solution quality}, respectively.
For MAX-$k$-CUT problems, the test bed is G-Set\footnote{available from \url{https://web.stanford.edu/~yyye/yyye/Gset/}}. More than half of the ratios between the best cut values achieved by our ODE approach and those by MOH are above 0.99. For star discrepancy calculation, the test set is a group of good lattice point (GLP) \addb{sets} adopted in \cite{winker_application_1997}. 
The ratios between the best lower bounds achieved by our ODE approach and those by TA\_improved are at least $0.91$.
More importantly, \adda{in both problems} the ODE approach requires much fewer iterations \adda{compared to MOH or TA\_improved} for each trial while the cost of each iteration is linear to the problem size \adda{across these three methods}. Its iteration steps over MOH's for MAX-$k$-CUT are about $1/10^{5}$ to $1/10^{4}$, and about {$1/10^{2}$ to $1/10$} over TA\_improved's for calculating the star discrepancy.

%\adda{under a certain constraint of runtime on the same platform}, 

To further demonstrate the capability of the proposed ODE approach, 
the numerical results obtained by using Gurobi to directly solve the MCPP instances 
are adopted as comparison for which the time limit of Gurobi is set to be the very runtime of the ODE approach. A detailed numerical comparison shows that, (1)
for MAX-$k$-CUT with $k=3,4,5$, in more than $80\%$ instances, the solutions produced by the ODE approach 
are better than or as good as Gurobi's; 
%Probably due to the enormous size of the Gurobi model for star discrepancy (which is much larger than the original problem), it requires much more time to presolve.
(2) for the star discrepancy problem, in about $17\%$ instances, Gurobi does not even provide a feasible solution under the time limitation, and in the remaining instances, it provides no better solutions than our ODE approach.

The paper is organized as follows. In Section~\ref{sec:alg} we show how we transform the MCPP problem~\eqref{eq:mcpp0} into the ODE system~\eqref{eq:main_ODE} step by step starting from the well-known assumption~\eqref{eq:B} of simulated annealing. In Section~\ref{sec:thy} we prove that a local optimum of MCPP~\eqref{eq:mcpp0} can be indeed obtained via an equilibrium point of the ODE system~\eqref{eq:main_ODE}. After that, Section~\ref{sec:detail} details the procedure of finding the equilibrium points by numerically integrating Eq.~\eqref{eq:main_ODE}. In Sections~\ref{sec:MAXkCUT} and \ref{sec:discrepancy},  we apply the proposed ODE approach into the MAX-$k$-CUT problem and approximating the star discrepancy, respectively. The paper is concluded in Section~\ref{sec:conclu} with a few remarks.

\section{The ODE approach: A continuous version of simulated annealing}
\label{sec:alg}

Under the multiple choice constraint in Eq.~\eqref{eq:mcpp0}, $x^{(j)}$ must be a standard unit vector in $\mathbb{R}^{d_j}$. Let $B_{d}=\{e_{i}^d\}_{i=1}^d\subset \mathbb{R}^{d}$ collect the standard base of $\mathbb{R}^d$ where $e^d_i$ denotes the $d$-dimensional unit vector the $i$-th entry of which equals $1$ and the rests are $0$s. The multiple choice constraint is then equivalent to $x^{(j)}\in B_{d_j},\ j=1,2,\dots,m$. Therefore, the feasible region for MCPP is
\begin{equation}\label{eq:X}
\mathcal{X}:=B_{d_1}\times B_{d_2}\times\dots\times B_{d_m}.
\end{equation}
We should mention that it is $d_j$, the cardinality of each $I_j$,  that determines the structure of the feasible region, rather than the particular devision,
and a compact notation, $B_d^m := \mathcal{X}$, will be used in Sections~\ref{sec:MAXkCUT} and \ref{sec:discrepancy} when $d_1=\dots=d_m=d$. 
Then MCPP~\eqref{eq:mcpp0} becomes $\min_{x\in \mathcal{X}} f(x)$ starting from which we explain how regard the ODE system \eqref{eq:main_ODE} 
as a continuous version of SA. The well-known SA method \cite{kirkpatrick_optimization_1983} approaches ``thermal equilibrium" by generating a Markov chain such that its stationary distribution is the Boltzmann distribution, namely
\begin{equation}\label{eq:B}
\Pr(X=x)\propto \exp(-E(x)/T),
\end{equation}
where $X\in \mathcal{X}$ is a random variable, and $E(x)$ denotes the energy at state $x$. 
We set $E(x):=f(x)$ for MCPP,
adopt a similar approach as SA but in a continuous, ODE-type way, by defining a continuous-time Markov chain, and then calculate the dynamics of expected value through the forward equation. 

To define a random variable of a continuous-time Markov chain, we need to determine two things: the state space of the variable
and the transition rate between states. For MCPP, the former is nothing but the feasible region \eqref{eq:X}. 
Let $X_t$ with $t\geq 0$ be the random variable. The state change of $X_t$ is allowed to happen only in a single $B_j$. That is, $X_t$ changes between $x=(x^{(1)}, x^{(2)}, \dots, x^{(m)})$ and $x'$, where there exists $j\in[m]$ and $i'\in [d_j]$ such that $x'=(x^{(\bar{j})},e_{i'}^{d_j})$. 
Here $x^{(\bar{j})}: = (x^{(1)}, \dots, x^{(j-1)}, x^{(j+1)}, \dots,  x^{(m)})$ denotes the $(n-d_j)$-dimensional vector by deleting $x^{(j)}$,
and $(x^{(\bar{j})},\xi): = (x^{(1)}, \dots, x^{(j-1)}, \xi, x^{(j+1)}, \dots,  x^{(m)})$ the $n$-dimensional vector by replacing 
$x^{(j)}$ with $\xi\in \mathbb{R}^{d_j}$. We are ready to determine the transition rate. 
Before that, we would like to {assume} that  $f(x^{(1)},x^{(2)},\dots,x^{(m)})$ is affine with respect to $x^{(j)}$ after fixing $x^{(\bar{j})}$ for all $j \in [m]$:
\begin{equation}\label{eq:affine}
f(x)=x^{(j)}\cdot \Phi^{(j)}(x^{(\bar{j})})+f((x^{(\bar j)},0)),
\end{equation}
where the dot $\cdot$ gives the standard inner product,
and the second RHS term is independent of $x^{(j)}$.
Eq.~\eqref{eq:affine} can be easily verified if noting: (1)  
every $(x^{(j)}_i)^k$ in $f$ can be replaced by $x^{(j)}_i$ because the power of $0$ or $1$ is equal to itself; 
(2) every monomial with divisor $x^{(j)}_i x^{(j)}_{i'}$, $i\neq i'$, vanishes; (3) the partial derivative $\Phi^{(j)}(x)$ is independent of $x^{(j)}$
and thus $\Phi^{(j)}(x^{(\bar{j})})=\Phi^{(j)}(x)$ is defined unambiguously. For $x=(x^{(\bar{j})},e_{i}^{d_j})$,
 $x'=(x^{(\bar{j})},e_{i'}^{d_j})$, $i\neq i'$, according to Eq.~\eqref{eq:affine}, we have 
 \begin{equation*}
 f(x) =f((x^{(\bar j)},0))+\Phi^{(j)}_i(x^{(\bar j)}), \quad
 f(x') = f((x^{(\bar j)},0))+\Phi^{(j)}_{i'}(x^{(\bar j)}), 
 \end{equation*}
 and 
\begin{equation}\label{eq:sum}
-f(x')+f(x)=-\Phi^{(j)}_{i'}(x^{(\bar j)})+\Phi^{(j)}_{i}(x^{(\bar j)}).
\end{equation}

Let $q\left(x\rightarrow x'\right)$ denote the transition rate from state $x$ to $x'$. It also means $(x,x')$-entry of the $Q$-matrix over $\mathbb{R}^{\mathcal{X}\times\mathcal{X}}$. The detailed balance condition reads
\[
\exp(-f(x)/T)q\left(x\rightarrow x'\right)=\exp(-f(x')/T)q\left(x'\rightarrow x\right),
\]
or
\begin{equation}
\label{eq:dtl_blc}
\frac{q\left(x\rightarrow x'\right)}{q\left(x'\rightarrow x\right)}=
\exp((-f(x')+f(x))/T)
\end{equation}
for non-zero entries of $q$ which implies $x\neq x'$. 
{To satisfy Eq.~\eqref{eq:dtl_blc}, combining it with Eq.~\eqref{eq:sum}}, we define 
\begin{equation}\label{eq:q}
q\left(x\rightarrow x'\right)=
\begin{cases} \sigma_{i'}\left(-\Phi^{(j)}(x^{(\bar{j})});1/T\right), &x=(x^{(\bar{j})},e_{i}^{d_j}), x'=(x^{(\bar{j})},e_{i'}^{d_j}), i\neq i',\\
0,& \text{otherwise.}
\end{cases}
\end{equation}
Using $x'^{(\bar j)}=x^{(\bar j)}$, we have 
\[
q\left(x'\rightarrow x\right)=\sigma_{i}\left(-\Phi^{(j)}(x'^{(\bar{j})});1/T\right)=\sigma_{i}\left(-\Phi^{(j)}(x^{(\bar{j})});1/T\right),
\]
and then Eq.~\eqref{eq:dtl_blc} can be readily verified.

%Therefore, $q$ satisfies
%\[
%\frac{q\left(x\rightarrow x'\right)}{q\left(x'\rightarrow x\right)}=\frac{\sigma_{i'}\left(-\Phi^{(j)}(x^{(\bar{j})});1/T\right)}{\sigma_{i}\left(-\Phi^{(j)}(x^{(\bar{j})});1/T\right)}=\frac{\exp(-\Phi^{(j)}_{i'}(x^{(\bar j)})/T)}{\exp(-\Phi^{(j)}_{i'}(x^{(\bar j)})/T)}.
%\]
%This is just equation (\ref{eq:dtl_blc}). It is worth noticing that defining the transition rate by partial derivative is possible since the objective function is affine to every $x^{(j)}$, which is a key property of MCPP.
%We use $x\sim x'$ to express $x\neq x'$ and $q\left(x\rightarrow x'\right)\neq 0$ shortly.

%Substitute $x$ by $(x^{(\bar{j})},e_{i}^{d_j})$. The aforementioned equation reads
%\[
%f(x^{(\bar{j})},e_{i}^{d_j})=f(x^{(\bar j)},0)+\Phi^{(j)}_i(x^{(\bar j)}).
%\]

%Therefore,
%\[\begin{aligned}
%\sigma_{i'}\left(-\Phi^{(j)}(x^{(\bar j)});1/T\right)&=\frac{\exp\left(-\Phi^{(j)}_{i'}(x^{(\bar j)})/T\right)}{\sum_{i=1}^{d_j}\exp\left(-\Phi^{(j)}_{i}(x^{(\bar j)})/T\right)}\\
%&=\frac{\exp\left(-(f(x^{(\bar j)},0)+\Phi^{(j)}_{i'}(x^{(\bar j)}))/T\right)}{\sum_{i=1}^{d_j}\exp\left(-(f(x^{(\bar j)},0)+\Phi^{(j)}_{i}(x^{(\bar j)}))/T\right)}\\
%&=\frac{\exp\left(-f(x')/T\right)}{\sum_{i=1}^{d_j}\exp\left(-f(x^{(\bar j)},e_{i}^{d_j})/T\right)},\\
%\end{aligned}\]
%Then we deduce that the transition rate matrix satisfies that
%\[
%\exp(-f(x)/T)q\left(x\rightarrow x'\right)=\exp(-f(x')/T)q\left(x'\rightarrow x\right),
%\]
%Obviously, this Markov chain is irreducible, so its stationary distribution is just Boltzmann distribution.

Next, {let us} calculate the dynamics of expected value. We denote the probability distribution at time $t$ by $p(x;t):=\Pr(X_t=x)$. Then it satisfies the forward equation
\begin{equation}\label{eq:forward}
\frac{\rmd p(x;t)}{\rmd t} = \sum_{x'\in \mathcal{X}, x'\neq x} q\left(x'\rightarrow x\right)p\left(x';t\right)-\sum_{x'\in \mathcal{X}, x'\neq x} q\left(x\rightarrow x'\right)p\left(x;t\right).
\end{equation}
Let $\mu:\mathbb{R}^n\rightarrow \mathbb{R}$ be an arbitrary vector function. 
%\sout{and $u_t=\left<\mu(X_t)\right>$ the mean of $\mu(X_t)$.} % 不用记号 u_t，避免和后面重复
From Eq.~\eqref{eq:forward}, we have
\begin{equation}
\label{eq:Mean}
\begin{aligned}
\frac{\rmd \left<\mu(X_t)\right>}{\rmd t} 
	&= \sum_{x\in \mathcal{X}} \mu(x)\frac{\rmd p(x;t)}{\rmd t}\\
	&= \sum_{x'\in\mathcal{X}} p(x';t)\sum_{x\neq x'}\mu(x)q(x'\rightarrow x)-\sum_{x\in\mathcal{X}}p(x;t)\sum_{x'\neq x}\mu(x)q(x\rightarrow x')\\
	&= \left<\sum_{x\neq X_t}(\mu(x)-\mu(X_t))q(X_t\rightarrow x)\right>,
\end{aligned}
\end{equation}
where the term in the summation only counts when both $\mu(x)\neq \mu(X_t)$ and $x\neq X_t$ hold.  
Let $m_t=\left<X_t\right>\in[0,1]^n$ be the expected value of $X_t$. Then $(m_t)^{(j)}_{i}=\sum_{x\in \mathcal{X}} x^{(j)}_i p(x;t)$ which assigns $\mu$ to $(\cdot)^{(j)}_i:x\mapsto x^{(j)}_i$
where there are only few $x$s satisfies these two conditions. Recalling the definition of $q$ in Eq.~\eqref{eq:q}, $x\neq x'$ iif $x$ and $x'$ differs in only one group $I_{k}$. In addition, we need $x^{(j)}_i\neq (x')^{(j)}_i$ then $k$ must be $j$. Therefore, $x=((X_t)^{(\bar{j})},e^{d_j}_{i'})$ for some $e^{d_j}_{i'}\neq (X_t)^{(j)}$. Accordingly, from Eqs.~\eqref{eq:Mean} and \eqref{eq:q}, we obtain
\begin{align}
\frac{\rmd (m_t)^{(j)}_i}{\rmd t} &= \left<\sum_{1\leq i'\leq d_j, e^{d_j}_{i'}\neq (X_t)^{(j)}}(x^{(j)}_i-(X_t)^{(j)}_i)q(X_t\rightarrow x)\right>\nonumber\\
	&= \left<\sum_{1\leq i'\leq d_j, e^{d_j}_{i'}\neq (X_t)^{(j)}}(x^{(j)}_i-(X_t)^{(j)}_i)\sigma_{i'}(-\Phi^{(j)}(x);1/T)\right>\nonumber\\
	&= \left<\sum_{1\leq i'\leq d_j}(x^{(j)}_i-(X_t)^{(j)}_i)\sigma_{i'}(-\Phi^{(j)}({x^{(\bar j)}});1/T)\right>\nonumber\\
	&= \left<\sum_{1\leq i'\leq d_j}x^{(j)}_i\sigma_{i'}(-\Phi^{(j)}({(X_t)^{(\bar{j})}});1/T)\right>-\left<\sum_{1\leq i'\leq d_j}(X_t)^{(j)}_i\sigma_{i'}(-\Phi^{(j)}({x^{(\bar j)}});1/T)\right>\nonumber\\
	&= \left<\sigma_{i}(-\Phi^{(j)}(X_t);1/T)\right>-\left<(X_t)^{(j)}_i\right>  \label{eq:mfa} \\
	&\approx  \sigma_{i}(-\Phi^{(j)}(m_t);1/T)-(m_t)^{(j)}_i, \label{eq:mfa1}
\end{align}
where we have applied $x^{(\bar{j})} = (X_t)^{(\bar{j})}$, and $x^{(j)}_i\neq 0$ iif $i'=i$ in the first RHS term of Eq.~\eqref{eq:mfa},  $\sum_{1\leq i'\leq d_j}\sigma_{i'}(-\Phi^{(j)}(x);1/T)=1$ in the second RHS term of Eq.~\eqref{eq:mfa},
as well as a rough approximation $\left<\sigma_{i}(-\Phi^{(j)}(X_t);1/T)\right>\approx\sigma_{i}(-\Phi^{(j)}(m_t);1/T)$ in the first RHS term of Eq.~\eqref{eq:mfa1}.
The ODE system~\eqref{eq:main_ODE} is manifest in Eq.~\eqref{eq:mfa1} after replacing $m_t$ with $y(t)\in(0,1)^n$. That is, {$y(t)$} is also an approximation of the dynamics of the expected value of a continuous-time Markov chain,
and by searching for its efficient numerical approximations, we are able to approximate the Markov chain efficiently.  Moreover, its equilibrium serves as the stable distribution of the Markov chain, namely, the Boltzmann distribution \eqref{eq:B}. When the temperature goes to zero, the distribution approaches the uniform distribution on the ground states, which may give the optimum solution. Therefore, the equilibrium gives a reasonable approximation of the optimum, the proof of which is left for Section \ref{sec:thy}.

%\add{
%Approximating the first term $\left<\sigma_{i}(-\Phi^{(j)}(X_t);1/T)\right>$ by $\sigma_{i}(-\Phi^{(j)}(\left<X_t\right>);1/T)=\sigma_{i}(-\Phi^{(j)}(m_t);1/T)$, we get
%\[
%\frac{\rmd (m_t)^{(j)}_i}{\rmd t} \approx \sigma_{i}(-\Phi^{(j)}(m_t);1/T)-(m_t)^{(j)}_i.
%\]
%Replacing $m_t$ by $y(t)\in(0,1)^n$ and $\approx$ by $=$, this becomes the ODE for $y(t)$ the same as Eq.~\eqref{eq:main_ODE}:
%\[
%\frac{\rmd{y^{(j)}_i}}{\rmd t}=-y^{(j)}_i+\sigma_i(-\Phi^{(j)}(y);1/T).
%\]
%}

Finally, $X_t\in \mathcal{X}$ means that the multiple choice constraint $\sum_{i=1}^{d_j} (X_t)^{(j)}_i=1$ for all $j\in [m]$ holds with probability $1$ and so does $m_t$. In fact, we claim that $y(t)$ defined in Eq.~\eqref{eq:main_ODE} also satisfy such constraint when the initial data $y(0)$ does, which can be readily seen from the sum of Eq.~\eqref{eq:main_ODE}: 
\begin{equation}
\frac{\rmd}{\rmd t} S_j(t)=-S_j(t)+1 \quad \Leftrightarrow \quad S_j(t) = (S_j(0)-1) \mathbbm{e}^{-t} + 1,
\end{equation}
where $S_j(t) =\sum_{i=1}^{d_j} y^{(j)}_i(t)$.

%\begin{equation}
%\frac{\rmd}{\rmd t}\sum_{i=1}^{d_j} y^{(j)}_i=-\sum_{i=1}^{d_j} y^{(j)}_i+1.
%\end{equation}

\begin{example}\rm
\label{ex1}
The ODE approach is applicable for the unconstrained situation where the multiple choice constraint in Eq.~\eqref{eq:mcpp0} is absent. By introducing extra $n$ Boolean variables $x_{n+1},\dots,x_{2n}$, the unconstrained problem $\min_{x\in\{0,1\}^n} f(x)$ can be reformulated into the following MCPP
\begin{equation}\label{eq:mcpp1}
\begin{aligned}
\min_{x\in\{0,1\}^{2n}} &&& f(x_1,x_2,\dots,x_n),\\
\text{s.t.} &&& x_i+x_{i+n}=1,\ i=1,2,\dots,n,
\end{aligned}
\end{equation}
where the index set $[2n]$ are divided into $n$ disjoint subsets: $I_j=\{j,j+n\}$, $d_j=2$, $j=1, 2, \dots, n$,
and $\Phi^{(j)}_1(x)=\frac{\partial f(x)}{\partial x_j}$, $\Phi^{(j)}_2(x)=0$. Then from Eq.~\eqref{eq:main_ODE}, 
we have that $y_j$ (i.e., $y^{(j)}_1$) satisfies 
\begin{equation}
\label{eq:ODE_uncons}
\frac{\rmd y_j}{\rmd t}=-y_j+\frac{1}{2}\left(\tanh\left(-\frac{1}{2T}\frac{\partial f(y)}{\partial x_j}\right)+1\right), \quad j=1, 2, \dots, n. 
\end{equation}
As $y_1^{(j)}(t)+y_{2}^{(j)}(t)=y_j(t)+y_{j+n}(t)=1$ for $t\geq 0$, we do not have to care about the dynamics of extra variable $y_{j+n}(t)$. Implementing a linear transformation $z_j = 2y_j-1$ in Eq.~\eqref{eq:ODE_uncons} yields a ODE sysytem for $z\in[-1,1]^n$: 
\[
\frac{\rmd z_j}{\rmd t}=-z_j+\tanh\left(-\frac{1}{2T}\frac{\partial f(\frac{1}{2}(z+1))}{\partial x_j}\right), \quad j=1, 2, \dots, n,
\]
which was also mentioned in \cite{peterson_neural_1988} for the unconstrained binary quadratic programming.
\end{example}

%\section{Analysis of the method}
\section{Numerical analysis}
\label{sec:thy}

Let $\bar{y}(T)$ be an equilibrium point of the ODE system~\eqref{eq:main_ODE}  under the parameter $T$, namely,
\begin{equation}\label{eq:ybar}
\bar y(T)^{(j)}=\sigma\left(-\Phi^{(j)}(\bar{y}(T));1/T\right), \ \forall\, j\in [m].
\end{equation}
We claim that $\bar{y}(T)$ approximates a local solution of Eq.~\eqref{eq:mcpp0} when $T$ is sufficiently small. This can be formally explained as follows.
As $\beta$ approaches $+\infty$, the limit of softmax function~\eqref{eq:soft}, denoted by  $\hat\sigma(z) =(\hat\sigma_i(z))$, is a kind of ``hard"max: 
\begin{equation}\label{eq:sigmahat}
\hat\sigma_i(z) = \lim_{\beta\rightarrow +\infty} \sigma_i(z;\beta)=
\begin{cases}
1/r, & z_i=\max \{z_1,z_2,\dots,z_d\},\\
0, & \text{otherwise},
\end{cases}
\end{equation}
where $r=|\{i|z_i=\max\{z_1,z_2,\dots,z_d\}\}|$ denotes the number of maximal entries of $z$. 
It can be easily observed that the range of $\hat{\sigma}(z)$, denoted by $\bar{B}_d$,  is a discrete set, i.e., 
\begin{equation}
\bar{B}_d=\left\{\frac{1}{|A|}\chi^A\in \mathbb{R}^d \middle| \chi^A_i=\left\{\begin{array} {l}
1,\,i\in A\\
0,\,i\notin A,
\end{array}\right. A\subseteq [d], \;\;A\neq \varnothing \right\}.  
\end{equation}
% 修改为 T 趋于 0
After assuming the limit {$\hat{y} := \lim\limits_{T\rightarrow 0^+} \bar{y}(T)$} exists, % 勘误
it can be formally deduced from Eq.~\eqref{eq:ybar} that % 只有极限存在的假设也不一定成立，所以用 formally
\begin{equation}\label{eq:equ_inf}
\hat y^{(j)}=\hat\sigma\left(-\Phi^{(j)}(\hat{y})\right),\ \forall\, j\in [m]. 
\end{equation}
That is, $\hat y^{(j)}$ belongs to $\bar{B}_{d_j}$, and thus 
\begin{equation}
\hat{y} \in \bar{\mathcal{X}}:=\bar{B}_{d_1}\times\bar{B}_{d_2}\times\dots \bar{B}_{d_m}.
\end{equation}
Combining Eqs.~\eqref{eq:sigmahat} and \eqref{eq:equ_inf}, we know that the nonzero entries in $\hat y^{(j)}$   correspond to the minimal entries of $\Phi^{(j)}(\hat{y})$, and
thus $\hat y^{(j)}$ minimizes the MCPP objective function $f((\hat{y}^{(\bar{j})},\xi))$ for  $\xi\in \mathbb{R}^{d_j}$ after fixing $\hat{y}^{(\bar{j})}$ in view of its affine property given in Eq.~\eqref{eq:affine}.
Actually, given Definition~\ref{def:local}, we are able to prove that $\hat{y}$ is indeed a local optimum of MCPP~\eqref{eq:mcpp0} in $\bar{\mathcal{X}}$ (see Proposition~\ref{prop:loc_opt}).

\begin{defn}[local optimality]
\label{def:local}
We call $x\in\bar{\mathcal{X}}$ a local optimum of $f$ in $\bar{\mathcal{X}}$,
if $\forall\, x'\in\bar{\mathcal{X}}$ satisfying $\exists\, j_0\in [m]$ such that $(x')^{(\bar{j_0})} = x^{(\bar{j_0})}$,
we have $f(x')\geq f(x)$. 
\end{defn}

% 重新修改了证明
%\begin{prop}
%\label{prop:loc_opt}
%Let $\hat{y}\in\bar{\mathcal{X}}$ satisfying Eq.~\eqref{eq:equ_inf}, then it is a local optimum of $f$ in $\bar{\mathcal{X}}$.
%\end{prop}

\begin{prop}
\label{prop:loc_opt}
The $n$-dimensional vector  $\hat{y}$ defined in Eq.~\eqref{eq:equ_inf} is a local optimum of MCPP~\eqref{eq:mcpp0} in $\bar{\mathcal{X}}$.
\end{prop}

\begin{proof}
Let $y'\in\bar{\mathcal{X}}$ satisfy $(y')^{(\bar{j_0})} = \hat{y}^{(\bar{j_0})}$. From Eqs.~\eqref{eq:sigmahat} and \eqref{eq:equ_inf}, we have 
\begin{equation}
\label{eq:loco_cond0}
\hat{y}^{(j_0)}_i=\frac{1}{r}, \;\; \Phi_i^{(j_0)} (\hat{y}^{\bar{(j_0)}})=\min_{1\leq k\leq d_{j_0}}\Phi_{k}^{(j_0)} (\hat{y}^{(\bar{j_0})}),\;\; \forall\ i\ \text{such that }\hat{y}^{(j_0)}_i\neq 0,
\end{equation}
where $r$ equals the number of nonzero entries of $\hat{y}^{(j_0)}$. 
Combining Eqs.~\eqref{eq:affine} and \eqref{eq:loco_cond0} yields 
\begin{equation*}
\begin{aligned}
f(y')-f(\hat{y})&=(y')^{(j_0)}\cdot\Phi^{(j_0)}((y')^{(\bar{j_0})})-\hat{y}^{(j_0)}\cdot\Phi^{(j_0)}(\hat{y}^{(\bar j_0)})\\
&=(y')^{(j_0)}\cdot\Phi^{(j_0)}(\hat{y}^{(\bar{j_0})})-\min_{1\leq k\leq d_{j_0}}\Phi_{k}^{(j_0)} (\hat{y}^{(\bar{j_0})}) \geq 0,
\end{aligned}
\end{equation*}
where the fact $(y')^{({j_0})}\in \bar{B}_{d_{j_0}}$ is used in the last inequality. 
\end{proof}

It should be noted that Eq.~\eqref{eq:equ_inf} does not hold in general since $\sigma(\cdot;\beta)$ does not converge to $\hat{\sigma}$ uniformly. However, in numerical experiments, the equilibriums $\bar{y}(T)$, which may not be close to a Boolean vector, are always very close to vectors in $\bar{\mathcal{X}}$ for sufficiently small $T$. Therefore, we can reasonably assume that $\bar{y}$ is an approximation of $\hat{y}$ and that $\sigma(\cdot;1/T)$ is an approximation of $\hat{\sigma}$. Then Eq.~\eqref{eq:ybar} approximates Eq.~\eqref{eq:equ_inf}. In fact, the theoretical results exists and in Proposition~\ref{prop:local} we give the analytical condition for the closeness between $\bar{y}$ and $\bar{\mathcal{X}}$ under which we can claim that Eq.~\eqref{eq:equ_inf} holds. Thus, by Proposition~\ref{prop:loc_opt}, $\hat{y}$ is a local optimum. Before that, we give some notations and definitions. 

Let $L$ be the maximal Lipschitz constant of all $\Phi_i^{(j)}$ with $j \in [m]$, $i\in [d_j]$. Namely, we have
\begin{equation}
\left|\Phi^{(j)}_i(x)-\Phi_i^{(j)}(x')\right|\leq L\|x-x'\|_\infty, \quad \forall \, j\in[m], \forall\, i\in [d_j], \forall\, x,x'\in[0,1]^n.
\end{equation}
Moreover, we define the ``minimal gap" of $\Phi$ as
\begin{equation}\label{mingap}
\mathsf{g} :=\min  \left\{\left|\Phi_{i'}^{(j)}(x)-\Phi_i^{(j)}(x)\right| \middle| \Phi_{i'}^{(j)}(x)\neq \Phi_i^{(j)}(x), x\in\bar{\mathcal{X}}, j\in[m], i', i\in [d_j]\right\}
\end{equation}
to measure how close can two different $\Phi^{(j)}_i$s in $\bar{\mathcal{X}}$ get,
and $\hat{d}:=\max\{d_1,d_2,\dots, d_m\}$.

\begin{prop}\label{prop:local}
Let $\bar{y}=\bar{y}(T)$ be an equilibrium of the ODE system \eqref{eq:main_ODE}, and assume that there exists $\hat{y}\in\bar{\mathcal{X}}$, whose distance to $\bar{y}$, denoted by $\varepsilon:=\|\hat{y}-\bar{y}\|_\infty>0$, satisfies
\begin{equation}\label{eq:equ_cond}
\hat{d}\varepsilon<\frac{1}{2},\;\;\; \frac{\varepsilon}{\ln(1/(\hat{d}\varepsilon)-1)}<\frac{T}{2L}, \;\;\; T\ln\frac{1+\hat{d}\varepsilon}{1-\hat{d}\varepsilon}+2L\varepsilon<\mathsf{g}.
\end{equation}
Then we have Eq.~\eqref{eq:equ_inf} holds and thus $\hat{y}$ must be a local optimum of $f$ in $\bar{\mathcal{X}}$.
\end{prop}

%By the definition of $\hat{\sigma}$, the first property implies that $\hat\sigma_i\left(-\Phi^{(j)}(\hat{y})\right)=1/r$ for all $i\in A_j$, and the second property implies that $\hat\sigma_i\left(-\Phi^{(j)}(\hat{y})\right)=0$ for all $i\notin A_j$. Since $\hat{y}\in \bar{\mathcal{X}}$, we have $\hat{y}^{(j)}=\chi^{A_j}/r$. Therefore, Eq.~\eqref{eq:equ_inf} holds.

%\begin{align}
%\Phi_i^{(j)} (\hat{y})&=\min_{1\leq k\leq d_{j}}\Phi_{k}^{(j)} (\hat{y}),\quad \forall\ i\in A_j,\label{eq:pry1}\\
%\Phi_i^{(j)} (\hat{y})&>\min_{1\leq k\leq d_{j}}\Phi_{k}^{(j)} (\hat{y}),\quad \forall\ i\notin A_j. \label{eq:pry2}
%\end{align}

\begin{proof} Let $A_j$ collect the nonzero entries of $\hat{y}^{(j)}$ and $r=|A_j|$. 
In order to obtain Eq.~\eqref{eq:equ_inf}, according to the definition of softmax function given in Eq.~\eqref{eq:soft}, it suffices to show for every $j \in [m]$, 
\begin{equation}\label{eq:pry}
\Phi_i^{(j)} (\hat{y}) 
\begin{cases}
=\min\limits_{1\leq k\leq d_{j}}\Phi_{k}^{(j)} (\hat{y}), & i\in A_j, \\ 
>\min\limits_{1\leq k\leq d_{j}}\Phi_{k}^{(j)} (\hat{y}), & i\notin A_j.
\end{cases}
\end{equation}
And, the conditions given in Eq.~\eqref{eq:equ_cond}  imply that  $\hat{d}\varepsilon<1$,  $\ln(1/(\hat{d}\varepsilon)-1)>0$, and $T\ln(1/(\hat{d}\varepsilon)-1)>2L\varepsilon$.

%We will prove for every $1\leq j\leq m$ that \sout{the property in Prop~\ref{prop:loc_opt}}
%\begin{equation}
%\label{eq:loco_cond}
%\Phi_i^{(j)} (\hat{y}^{\bar{(j)}})=\min_{1\leq k\leq d_{j}}\Phi_{k}^{(j)} (\hat{y}^{(\bar{j})}),\quad \forall\ i\ \text{such that }\hat{y}^{(j)}_i\neq 0.
%\end{equation}

First, we prove that  $\Phi_{i_1}^{(j)}(\hat y)=\Phi_{i_2}^{(j)}(\hat y)$, $\forall\ i_1, i_2\in A_j$. Using the fact that $\bar{y}$ is an equilibrium, we arrive at 
\begin{equation}\label{eq:ln}
\ln\frac{\bar{y}_{i_1}^{(j)}}{\bar{y}_{i_2}^{(j)}}=\frac{1}{T}\left(-\Phi_{i_1}^{(j)}(\bar{y})+\Phi_{i_2}^{(j)}(\bar{y})\right).
\end{equation}
On the other hand, the distance between $\hat{y}$ and $\bar{y}$ provides a limit for every entry of $\bar{y}^{(j)}$. By $|\hat{y}^{(j)}_i-\bar{y}^{(j)}_i|\leq \|\hat{y}-\bar{y}\|_\infty=\varepsilon$, we get
\begin{equation}\label{eq:yin}
\bar{y}^{(j)}_i \in 
\begin{cases}
\left[\frac{1}{r}-\varepsilon,\frac{1}{r}+\varepsilon\right], & i\in A_j, \\
[0,\varepsilon], & i\notin A_j.
\end{cases}
\end{equation}
Combining Eqs.~\eqref{eq:equ_cond}, \eqref{eq:ln} and \eqref{eq:yin} together and using $r\leq \hat{d}$ lead to 
an estimate,
\[
\left|\Phi_{i_1}^{(j)}(\bar{y})-\Phi_{i_2}^{(j)}(\bar{y})\right|=T\left|\ln\frac{\bar{y}_{i_1}^{(j)}}{\bar{y}_{i_2}^{(j)}}\right|\leq  T\ln\frac{\frac{1}{r}+\varepsilon}{\frac{1}{r}-\varepsilon} \leq T\ln\frac{1+\hat{d}\varepsilon}{1-\hat{d}\varepsilon},\;\;
\forall\ i_1,i_2\in A_j,
\]
and thus we have 
\begin{align*}
\left|\Phi_{i_1}^{(j)}(\hat y)-\Phi_{i_2}^{(j)}(\hat y)\right|&\leq \left|\Phi_{i_1}^{(j)}(\hat y)-\Phi_{i_1}^{(j)}(\bar{y})\right|+\left|\Phi_{i_1}^{(j)}(\bar{y})-\Phi_{i_2}^{(j)}(\bar{y})\right|+\left|\Phi_{i_2}^{(j)}(\hat y)-\Phi_{i_2}^{(j)}(\bar{y})\right|\\
	&\leq L\varepsilon+T\ln\frac{1+\hat{d}\varepsilon}{1-\hat{d}\varepsilon}+L\varepsilon< \mathsf{g}, 
\end{align*}
which directly implies $\Phi_{i_1}^{(j)}(\hat y)=\Phi_{i_2}^{(j)}(\hat y)$ for all $i_1, i_2\in A_j$,  because
$\mathsf{g}$ defined in Eq.~\eqref{mingap} gives the minimal distance between two different $\Phi^{(j)}_i$s in $\bar{\mathcal{X}}$. 

In order to verify Eq.~\eqref{eq:pry}, the rest is to prove that $\Phi_{i'}^{(j)}(\hat y)>\Phi_{i_1}^{(j)}(\hat y)$ for all $i'\not \in A_j$. According to Eqs.~\eqref{eq:equ_cond}, \eqref{eq:ln} and \eqref{eq:yin} and the fact that $r\leq \hat{d}$, we have 
\[
\Phi_{i'}^{(j)}(\bar{y})-\Phi_{i_1}^{(j)}(\bar{y})=T\ln\frac{\bar{y}_{i_1}^{(j)}}{\bar{y}_{i'}^{(j)}} \geq T\ln\frac{\frac{1}{r}-\varepsilon}{\varepsilon} \geq T\ln\left(\frac{1}{\hat{d}\varepsilon}-1\right)> 2L\varepsilon,
\]
and then 
\begin{align*}
\Phi_{i'}^{(j)}(\hat y)-\Phi_{i_1}^{(j)}(\hat y)&=\left(\Phi_{i'}^{(j)}(\bar{y})-\Phi_{i_1}^{(j)}(\bar{y})\right)+\left(\Phi_{i'}^{(j)}(\hat y)-\Phi_{i'}^{(j)}(\bar{y})\right)-\left(\Phi_{i_1}^{(j)}(\hat y)-\Phi_{i_1}^{(j)}(\bar{y})\right)\\
	&>2L\varepsilon-\left|\Phi_{i'}^{(j)}(\hat y)-\Phi_{i'}^{(j)}(\bar{y})\right|-\left|\Phi_{i_1}^{(j)}(\hat y)-\Phi_{i_1}^{(j)}(\bar{y})\right|\\
	&\geq 2L\varepsilon-L\varepsilon-L\varepsilon = 0.
\end{align*}
The proof is completed.
%Thus we prove that
%\[
%\Phi_i^{(j)} (x)=\min_{1\leq k\leq d_j}\Phi_{k}^{(j)} (W_{\bar{j}}),\quad \forall i\in A_j,
%\]
%and $\hat{y}$ is local optimal.
\end{proof}

\begin{rmk}\rm
\label{rmk:cond}
There exists a more concise condition for $\varepsilon$: 
\begin{equation}
\varepsilon <\min\left\{\frac{1}{4\hat{d}},\frac{T}{2L},\frac{\mathsf{g}}{3\hat{d}T+2L}\right\},
\end{equation}
from which the three conditions given in Eq.~\eqref{eq:equ_cond} can be readily derived by direct relaxations. 
That is, for a given $T$, a sufficiently small $\varepsilon$ guarantees an equilibrium point of the ODE system \eqref{eq:main_ODE} gives a local optimum of MCPP~\eqref{eq:mcpp0} in $\bar{\mathcal{X}}$. However, we cannot theoretically assure that $\varepsilon$ will be small enough as $T$ approaches $0$ at current stage.
\end{rmk}

%\sout{
%Thus, in the rounding procedure of the algorithm, we first round the equilibrium $y$ into $\hat{y}\in \bar{\mathcal{X}}$, then for $j=1,2,\dots,m$, arbitrarily choose an $i_j$ such that $\hat{y}_{i_j}^{(j)}>0$, and let $\bar{y}\in\mathbb{R}^n$ with $\bar{y}_{i_j}^{(j)}=1$. By the previous deduction, if we choose $T_M$ small enough,  $\hat{y}$ is a local optimum and $f(\hat{y})=f(\bar{y})$, but we can not ensure that $\bar{y}$ is also locally optimal since generally $\nabla f(\hat{y})\neq \nabla f(\bar{y})$. 
%}

%\section{Details and variations of the method}
\section{Implementation details}
\label{sec:detail}
The biggest benefit of transforming MCPP~\eqref{eq:mcpp0} into the ODE system~\eqref{eq:main_ODE} is that 
various numerical ODE solvers and related techniques come into play for complex combinatorial problems. 
From Proposition~\ref{prop:local}, we need to find the equilibrium of Eq.~\eqref{eq:main_ODE} quickly rather than to know how we approach it.
That is, we do not have to know the dynamics evolution quite precisely. To this end, we adopt variable time step forward Euler scheme for accelerating the calculation. Besides, also inspired by SA, we use an infinite temperatures sequence (mostly it is decreasing) $T_1,T_2,\dots,$ and numerically solve Eq.~\eqref{eq:main_ODE} with $T:=T_s,\ s=1,2,\dots,$ until we find an equilibrium $\bar{y}(T_s)$. Except for the first {numerical integration} with $T=T_1$, the initial value for the ODE system~\eqref{eq:main_ODE} with $T=T_s$ can be set to be the previous equilibrium $\bar{y}(T_{s-1})$. When the distance $\varepsilon:=\min_{y\in\bar{\mathcal{X}}} \|y-\bar{y}(T)\|_\infty$ is less than a prescribed accuracy tolerance $\varepsilon_0$, we stop evolving the ODE system and turn to rounding procedure.

% 重新组织了初始温度这段，增加了温度不能太小的说明以及如何找到合适的初始温度
%--------rewrite--------

It should be noted that there is a suitable range for $T_1$ due to the following reasons.
\begin{enumerate}[(1)]
\item $T_1$ should not be too large or the diversity of solutions may be insufficient. To see this, consider the extreme case by formally letting $T=+\infty$ and replacing $1/T$ by $0$ in Eq.~\eqref{eq:ybar}. We get $\bar{y}(+\infty)^{(j)}=\bm 1_{d_j}/d_j$, where $\bm 1_{d_j}$ is a $d_j$-dimensional vector with every entry identical to $1$. This uninformative equilibrium is independent of the initial point $y_0$. A similar phenomenon happens when $T_1$ is large enough, in which the numerical experiments show that we can only find one equilibrium $\bar{y}(T_1)$ for different $y_0$s (and it is close to $\bar{y}(+\infty)$). Then the subsequent evolutions of $y$ will be similar. Sometimes the method even produces the same solutions. 
\item  $T_1$ should not be too small either. Otherwise the ODE will act like a greedy local search. This phenomenon is accessible noticing that the SA method is more like greedy algorithm when the temperature becomes smaller and so does the ODE approach.
\end{enumerate}
In practice, it is not necessary to determine the range precisely. To find a reasonable assign for $T_1$, we first choose a small value then double it repeatedly until $2T_1$ leads to uninformative equilibriums while $T_1$ does not. Thus the above two requirements are satisfied.

\subsection{Choice of the initial data}
\label{sec:rand_init}
As mentioned in Section~\ref{sec:intro},  the initial data of ODE $y_0$ should satisfy the ``continuous" multiple choice constraint, $\sum_{i=1}^{d_j} (y_0)_i^{(j)} = 1$ for all $j\in[m]$. Namely, $(y_0)^{(j)}$ belongs to the standard $(d_j-1)$-simplex $\Delta^{d_j-1}:=\{z\in\mathbb{R}^{d_j}|z\geq 0, \sum_{i=1}^{d_j} z_i = 1\}$ .
% The numerical experiments show that when $(y_0)^{(j)}$ is close to the center of $\Delta^{d_j-1}$, the diversity of the solutions is limited.
Therefore, we generate $(y_0)^{(j)}$s by sampling  % 或者把 $(y_0)^{(j)}$s 换成这个 "the entries in $I_j$ of initial values"?
from a distribution on this simplex and then combine $(y_0)^{(j)},\ j=1,2,\dots,m$ into one $y_0$. 
In practice, we choose the distribution to be the symmetric Dirichlet distribution with the concentration parameter equal to $0.01$, and it is more likely to produce values most entries of which are close to $0$.  This may lead to a larger expected distance between each pair of $y_0$s and thus a more thorough exploration of the solution space of the ODE system.

%help us explore the solution space of the ODE more thoroughly.

\subsection{Variable time step}
The forward Euler (FE) scheme for an autonomous system $\frac{\rmd y}{\rmd t}=F(y)$ reads 
\begin{equation}\label{fe}
y^{k}=y^{k-1}+h^{k-1} F^{k-1},
\end{equation}
where $F:\mathbb{R}^n\rightarrow \mathbb{R}^n$ is a vector function, $y^k$ denotes the numerical {approximation} of $y(t)$ at time $t^k$,
$h^{k-1}:=t^{k}-t^{k-1}$ is the time step, and $F^k:=F(y^k)$.
To estimate the local truncation error of Eq.~\eqref{fe}, which is dominated by $(h^{k-1})^2$ when $h^{k-1}$ is small enough, we adopt a method similar to the Richardson extrapolation \cite{richardson_approximate_1911}. 
Let $(h^{k-1})^2c$ denote the error. Assuming that the vector $c$ keeps almost unchanged for some successive time steps and $h^{k-1}=h^{k-2}$ for even $k$, we have another approximation to $y(t^k)$, $\tilde{y}^k:=y^{k-2}+(h^{k-1}+h^{k-2}) F^{k-2}=y^{k-2}+2h^{k-1} F^{k-2}$, the error of which is $4(h^{k-1})^2c$, while Eq.~\eqref{fe} has the error 
$(h^{k-1})^2c+(h^{k-2})^2c=2(h^{k-1})^2c$ staring from the same $y^{k-2}$.
That is, we may use $(\tilde{y}^k-y^k)/2\approx (h^{k-1})^2c$ to estimate the truncation error in practice. 
Specifically, we adopt a {tolerance} $\Theta>0$ for $\theta^k:=\|\tilde{y}^k-y^k\|_2$ and a step adjust ratio $\rho>1$. Since the error for the $k$-th step is proportion to the square of step size $h^{k-1}$ and can be well approximated by $\theta^k/2$, we are able to maintain it within a specific range according to $\Theta$ and $\rho$ by a two-way adjustment of $h^k$ as follows. Reduce the time step $h^{k+1}=h^{k}=h^{k-1}/\rho$ when $\theta^k>\Theta \rho^2$, which may avoid the error increase by decreasing the time step; Increase the time step $h^{k+1}=h^{k}=\rho h^{k-1}$ when $\theta^k<\Theta/ \rho^2$, which may exploit the maximum large time step allowed by a small error. {If the errors make a small change in successive two steps, then $\theta^k$ will not be far from $[\Theta/\rho^2,\rho^2\Theta]$}. Thus, in doing so, the errors can be controlled around a desired accuracy of $\Theta/2$, and few of them will be outside the interval $[\Theta/2\rho^2,\rho^2\Theta/2]$.

\subsection{Rounding procedure}

Proposition~\ref{prop:local} requires $\varepsilon=\|\hat{y}-\bar{y}(T)\|_\infty$ to be small enough. It is natural to think of letting $\hat{y}$ reach the minimum among $\bar{\mathcal{X}}$. 
However, a naive implementation of this procedure is impractical since each $B_{d_j}$ contains $2^{d_j}-1$ elements and thus $|\bar{\mathcal{X}}|=\prod_{j=1}^m (2^{d_j}-1)$.
Instead, we only generate an appropriate $\hat{y}$ as follows. For each $\bar{y}^{(j)}$, let $\eta_j = \max_{i\in[d_j]}\{\bar{y}_i^{(j)}\}$ and $r_j=\left\lfloor 1/\eta_j+1/2\right\rfloor$. Then consider the largest $r_j$ entries of $\bar{y}^{(j)}$ and record the indices in set $A_j\subseteq [d_j]$. This always can be done after noting $\eta_j\geq 1/d_j$. Let $\hat{y}^{(j)}=\chi^{A_j}/r_j$ for each $j\in[m]$ then $\hat{y}^{(j)}\in\bar{B}_{d_j}$ as $|A_j|=r_j\leq d_j$, and we achieve $\hat{y}\in\bar{\mathcal{X}}$. It should be pointed out that if there exists a $r_j'$-element set $A_j'\subseteq [d_j]$ such that $\|\bar{y}^{(j)}-\chi^{A_j'}/r_j'\|_{\infty}$ is small enough, $1/\eta_j$ will be close to $r_j'$, thus we choose the right $r_j=r_j'$. Moreover, the largest $r_j$ entries of $\bar{y}^{(j)}$ will be close to $1/r_j'$, which form $A_j'$ and are also chosen by $A_j$. This explains why $\hat{y}^{(j)}$ may serve as a good approximation for $\bar{y}^{(j)}$.

At the final step, we have to get a Boolean vector $x$ from $\hat{y}$. First, we let $x\leftarrow \hat{y}$. Then sequentially change $x^{(j)}$ into a standard unit vector. This process should be done greedily. More precisely, we choose and $i\in [d_j]$ such that minimize $\Phi_i^{(j)}(x^{(\bar{j})})$ and let $x^{(j)}\leftarrow e_i^{d_j}$. This will not increase $f(x)$. When $x$ is locally optimal, the procedure ends. Thus $f(x)$ is no larger than $f(\hat{y})$.

\subsection{Cost analysis}
\label{subsec:cost}

The main cost of each FE step in Eq.~\eqref{fe} is calculating $F(y)$, and according to Eq.~\eqref{eq:main_ODE}, the cost of $F(y)$ consists of calculating $\Phi(y)$, $m$ softmax functions with $\Phi(y)^{(j)}$ as input and $\mathcal{O}(n)$ basic operations. The softmax functions requires another $\mathcal{O}(n)$ operations. Therefore, the computational complexity of the gradient function may dominate the cost and it is necessary to analyze it respectively for different problems in practice.

\section{Application to MAX-$k$-CUT}
\label{sec:MAXkCUT}

As a classical graph optimization problem, the MAX-$k$-CUT problem wants a $k$-division 
of the vertex set $V$ for a given edge-weighted graph $G=(V, E, W)$ such that the sum of weights across any two subsets is maximized. Let $V=\{v_i\}_{i=1}^{|V|}$,  $W=(w_{ij})_{|V|\times |V|}$ and $V_1$, $V_2,\dots$, $V_k$ denote the subsets after division. Precisely, MAX-$k$-CUT maximizes the following cut function
\begin{equation}\label{prob:MAXkCUT}
\mathrm{cut}(V_1,V_2,\dots,V_k):=\sum_{v_i\in V_r, v_j\in V_s, r\neq s, i<j} w_{ij}.
\end{equation}
It reduces to the famous MAX-CUT problem for $k=2$, one of the $21$ Karp's NP-complete problems \cite{karp_reducibility_1972}. MAX-$k$-CUT is widely studied and there are a number of algorithms for finding 
its approximating solutions,  including the heuristic methods, see e.g. \cite{ma_multiple_2017} and reference therein,  and continuous algorithms \cite{goemans_improved_1995, shao_simple_2019, burer_rank-two_2002}.

%These algorithms are divided into mainly two classes, discrete and continuous. Discrete algorithms often use heuristic strategies to iteratively modify the variable and explore the feasible region. See, e.g., \cite{ma_multiple_2017} and \cite{marti_advanced_2008}. In continuous algorithms, the introduction of continuous variables gave us a different view of problems. These include Goemans-Williamson SDP relaxation\cite{goemans_improved_1995,goemans_approximation_2004} and Lov\"asz extension\cite{shao_simple_2019}. 
%It is well known that MAX-CUT problem can be formulated as unconstrained binary quadratic programming. In what follows, we show that MAX-$k$-CUT problem can be tranformed into MCPP form.

We will use the proposed ODE approach to solve the MAX-$k$-CUT problem and need to transform it into MCPP at the first step. This can be readily implemented by setting $n \leftarrow k|V|$, $m \leftarrow |V|$, $d_j \leftarrow k$ for any $j\in[m]$ in Eqs.~\eqref{eq:mcpp0} and \eqref{eq:X}, and the feasible region for MAX-$k$-CUT turns out to be  $B_k^{|V|}$.
Accordingly, the belonging of each vertex $v_i$ is mapped to an element in $B_k$ as follows: $v_i\in V_r$ if and only if $x^{(i)}=e^{k}_r$; $v_i$ and $v_{j}$ belong to the same subset if and only if $(x^{(i)})\transpose x^{(j)}=1$. That is, the corresponding MCPP objective function becomes
\begin{equation}
\label{eq:MkC_MCPP}
f(x) = - \mathrm{cut}(V_1,V_2,\dots,V_k)=- \sum_{i<j} w_{ij}\left(1-(x^{(i)})\transpose x^{(j)}\right).
\end{equation}
The minus sign before the cut function converts it into a minimization problem to fit the form of \eqref{eq:mcpp0}. 
For the sake of comparison, we still regard $-f$ as the result \adda{cut} value in the rest of the paper. % cut value 这个词组在后面直接引入有点突兀，在这里快速引入
Then our ODE approach can be used directly to solve the MAX-$k$-CUT problem. 
We implement it with MATLAB R2021a and run it on AMD Ryzen 1950X with 64GB RAM.

%The MOH algorithm \cite{ma_multiple_2017} is adopted there to as the reference \adda{of the solution quality}.

\adda{Due to the differences in programming languages, implementation and platforms, making a fair comparison of runtime is quite difficult. Therefore, we focus on the comparison of solution quality where the cut values produced by MOH \cite{ma_multiple_2017} are adopted as the reference.} Tables~\ref{tab:MAXCUT_res}--\ref{tab:MAXkCUT_time} present the numerical results for $k=2,3,4,5$ on G-Set which includes $71$ randomly generated graph with $800$ to $20000$ vertices. \adda{So there are $4\times 71$ instances in all.} 
  \adda{We run 100 independent trials for each instance and record the best cut value among all the results} in Tables~\ref{tab:MAXCUT_res}--\ref{tab:MAX5CUT_res}, see the columns headed by ``ODE". \adda{The initial value of each trial is generated randomly as described in Section~\ref{sec:rand_init}.} 
The \adda{trials} are simply parallelized with MATLAB's \texttt{parfor}.
Other parameters are $T_1=3$, $T_s=\gamma^{s-1}T_1,s=2,3,\dots$, $\gamma = 0.95$, $\varepsilon_0=1\times 10^{-3}$, $\Theta=1\times 10^{-6}n$ and $\rho = 1.1$. 
{$T_1$ is determined according to the routine described in Section~\ref{sec:detail}, and $\{T_s\}$ is a geometric progression since it is a simple sequence converging to zero.}
For comparison, Tables~\ref{tab:MAXCUT_res}--\ref{tab:MAX5CUT_res} also calculate the ratio between the \adda{best} cut value achieved by our ODE approach in those $100$ runs and that by MOH for each instance, and {one can} find that, for all four problems, such ratios are larger than $0.99$ for at least $38$ instances.  For the remaining graphs of G-Set,  the ODE method performs diversely. 

For $k=2$, the ratios are at least $0.96$ and in only one instance is less than $0.97$ (see G64 in Table~\ref{tab:MAXCUT_res}).  For $k=3$, the ratios are all above $0.97$. However, there are $5$ instances for $k=4$ the ratios of which are in $(0.95,0.97]$ (see G18, G19, G20, G56, G61 in Table~\ref{tab:MAX4CUT_res}), and there are $6$ ratios in this interval for $k=5$ (see G19, G20, G21, G40, G56, G61 in Table~\ref{tab:MAX5CUT_res}). 
 The worst solution quality happens in solving the MAX-$5$-CUT problem and the ratio decreases to $0.9472$ for G18 and it is the only one less than $0.95$ in Table~\ref{tab:MAX5CUT_res}.  For $k=2, 4, 5$, the ratios reach $0.98$ in at least $56$ instances, but only $46$ ratios reach $0.98$ for $k=3$. 

In order to further show the overall performance in the solution quality achieved by the ODE approach,  
\adda{for all $(4\times 71) \times 100$ trials, we calculate the ratios between the cut values of the ODE approach and the reference values obtained from MOH, and plot the histogram of all these ratios in Figure~\ref{fig:MAXkCUT_qlt}.}
\adda{It can be observed that,} in more than $97\%$ trials, the ratios are greater than $0.95$,
thereby implying that the approximate cuts produced by the ODE approach are of comparable quality to MOH.
Moreover, we compare our ODE approach's solutions with those obtained by using Gurobi 10.0.1 to directly solve MCPP in the same time duration (see the columns headed by ``Gurobi" in Tables~\ref{tab:MAXCUT_res}--\ref{tab:MAX5CUT_res}). 
We find that, for $k=2$, in more than $70\%$ instances, ODE's solutions are better than or as good as Gurobi's,
and the proportion gets larger and reaches $80\%$ for $k=3,4,5$.

It should be noted that, given an initial data, the ODE approach is deterministic,
and there is no heuristic operations dedicated to MAX-$k$-CUT \addb{since the implementation  and the running parameters strictly follow the general guidelines detailed in Section~\ref{sec:detail}}. 
Hence, we may conclude that, 
the proposed ODE approach can produce relatively good results without any ad hoc designs.

Next, we are going to show the efficiency of our ODE approach. Let $w_{\text{tot}}=\sum_{i<j} w_{ij}$ represent the total weight of all edges and $P=(p_{ij})\in\{0,1\}^{k\times |V|}$
a matrix satisfying $p_{ij}=x^{(j)}_i$. It is easy to check that
\begin{equation}
f(x)=-w_{\text{tot}}+\frac{1}{2}\trace(PWP\transpose)\Rightarrow \frac{\partial f}{\partial P}=PW,
\end{equation}
where we use the denominator layout, thereby implying that the complexity of calculating $\nabla f$ is   $\mathcal{O}(k|E|)$. To obtain this order we have used that $W$ is a sparse matrix with $2|E|$ nonzero elements and calculating each row of $PW$ requires $\mathcal{O}(|E|)$ operations. By the observation in Section~\ref{subsec:cost}, we deduce that the time complexity of each FE step is $\mathcal{O}(k(|E|+|V|))$. 

Table~\ref{tab:MAXkCUT_time} reports the approximate total steps and time needed by 
running the ODE approach once and shows clearly that the computing cost is not expensive \adda{compared to MOH.}
\adda{In MOH, each iteration step has a time complexity of $\mathcal{O}(|E|+k|V|)$ and the total iteration steps may be up to $10^6 \sim 10^8$ \cite{ma_multiple_2017}.}

\begin{table}[p]
\footnotesize
\centering
\caption{\small Numerical cut values for MAX-CUT.}
\label{tab:MAXCUT_res}
\begin{tabular}{l|N{2.5em}|N{2.5em}|N{2.5em}|N{3em}||l|N{2.5em}|N{2.5em}|N{2.5em}|N{3em}}
\toprule
\multicolumn{1}{l|}{Graph} & \multicolumn{1}{l|}{ODE}   & \multicolumn{1}{l|}{MOH} & \multicolumn{1}{l|}{Ratio} & \multicolumn{1}{l||}{Gurobi}
& \multicolumn{1}{l|}{Graph}  & \multicolumn{1}{l|}{ODE}  & \multicolumn{1}{l|}{MOH} & \multicolumn{1}{l|}{Ratio} & \multicolumn{1}{l}{Gurobi} \\
 \midrule

G1  & 11623 & 11624 & 0.9999  & 11624 & G37 & 7629 & 7691 & 0.9919  & 7535 \\
G2  & 11612 & 11620 & 0.9993  & 11620 & G38 & 7624 & 7688 & 0.9917  & 7543 \\
G3  & 11621 & 11622 & 0.9999  & 11622 & G39 & 2355 & 2408 & 0.9780  & 2250 \\
G4  & 11646 & 11646 & 1.0000  & 11646 & G40 & 2342 & 2400 & 0.9758  & 2198 \\
G5  & 11626 & 11631 & 0.9996  & 11631 & G41 & 2333 & 2405 & 0.9701  & 2166 \\
G6  & 2173 & 2178 & 0.9977  & 2178 & G42 & 2429 & 2481 & 0.9790  & 2307 \\
G7  & 1999 & 2006 & 0.9965  & 2006 & G43 & 6645 & 6660 & 0.9977  & 6225 \\
G8  & 2005 & 2005 & 1.0000  & 2005 & G44 & 6640 & 6650 & 0.9985  & 6173 \\
G9  & 2045 & 2054 & 0.9956  & 2038 & G45 & 6638 & 6654 & 0.9976  & 6147 \\
G10 & 1993 & 2000 & 0.9965  & 1992 & G46 & 6640 & 6649 & 0.9986  & 6282 \\
G11 & 554 & 564 & 0.9823  & 564 & G47 & 6644 & 6657 & 0.9980  & 6200 \\
G12 & 548 & 556 & 0.9856  & 556 & G48 & 6000 & 6000 & 1.0000  & 6000 \\
G13 & 576 & 582 & 0.9897  & 582 & G49 & 6000 & 6000 & 1.0000  & 6000 \\
G14 & 3049 & 3064 & 0.9951  & 3015 & G50 & 5880 & 5880 & 1.0000  & 5880 \\
G15 & 3032 & 3050 & 0.9941  & 2993 & G51 & 3823 & 3848 & 0.9935  & 3773 \\
G16 & 3027 & 3052 & 0.9918  & 2989 & G52 & 3825 & 3851 & 0.9932  & 3769 \\
G17 & 3030 & 3047 & 0.9944  & 2994 & G53 & 3829 & 3850 & 0.9945  & 3781 \\
G18 & 979 & 992 & 0.9869  & 918 & G54 & 3820 & 3852 & 0.9917  & 3782 \\
G19 & 891 & 906 & 0.9834  & 853 & G55 & 10211 & 10299 & 0.9915  & 9906 \\
G20 & 931 & 941 & 0.9894  & 869 & G56 & 3954 & 4016 & 0.9846  & 3851 \\
G21 & 914 & 931 & 0.9817  & 880 & G57 & 3414 & 3494 & 0.9771  & 3494 \\
G22 & 13325 & 13359 & 0.9975  & 11180 & G58 & 19096 & 19288 & 0.9900  & 18928 \\
G23 & 13337 & 13344 & 0.9995  & 12368 & G59 & 5911 & 6087 & 0.9711  & 5716 \\
G24 & 13325 & 13337 & 0.9991  & 11008 & G60 & 14089 & 14190 & 0.9929  & 13989 \\
G25 & 13321 & 13340 & 0.9986  & 12572 & G61 & 5678 & 5798 & 0.9793  & 5590 \\
G26 & 13308 & 13328 & 0.9985  & 11132 & G62 & 4766 & 4868 & 0.9790  & 4872 \\
G27 & 3325 & 3341 & 0.9952  & 2515 & G63 & 26799 & 27033 & 0.9913  & 26545 \\
G28 & 3286 & 3298 & 0.9964  & 2174 & G64 & 8460 & 8747 & 0.9672  & 8239 \\
G29 & 3379 & 3405 & 0.9924  & 2497 & G65 & 5444 & 5560 & 0.9791  & 5562 \\
G30 & 3389 & 3413 & 0.9930  & 2442 & G66 & 6208 & 6360 & 0.9761  & 6364 \\
G31 & 3299 & 3310 & 0.9967  & 2546 & G67 & 6802 & 6942 & 0.9798  & 6950 \\
G32 & 1384 & 1410 & 0.9816  & 1410 & G70 & 9482 & 9544 & 0.9935  & 9546 \\
G33 & 1362 & 1382 & 0.9855  & 1382 & G72 & 6840 & 6998 & 0.9774  & 7008 \\
G34 & 1362 & 1384 & 0.9841  & 1384 & G77 & 9712 & 9928 & 0.9782  & 9940 \\
G35 & 7628 & 7686 & 0.9925  & 7544 & G81 & 13724 & 14036 & 0.9778  & 14058 \\
G36 & 7612 & 7680 & 0.9911  & 7538 &     &     &     &     &  \\

\bottomrule
\end{tabular}
\end{table}

\begin{table}[p]
\footnotesize
\centering
\caption{\small Numerical cut values for MAX-$3$-CUT.}
\label{tab:MAX3CUT_res}
\begin{tabular}{l|N{2.5em}|N{2.5em}|N{2.5em}|N{3em}||l|N{2.5em}|N{2.5em}|N{2.5em}|N{3em}}
\toprule
\multicolumn{1}{l|}{Graph} & \multicolumn{1}{l|}{ODE}   & \multicolumn{1}{l|}{MOH} & \multicolumn{1}{l|}{Ratio} & \multicolumn{1}{l||}{Gurobi}
& \multicolumn{1}{l|}{Graph}  & \multicolumn{1}{l|}{ODE}  & \multicolumn{1}{l|}{MOH} & \multicolumn{1}{l|}{Ratio} & \multicolumn{1}{l}{Gurobi} \\
 \midrule

G1  & 15158 & 15165  & 0.9995  & 14863 & G37 & 9968 & 10052  & 0.9916  & 9893 \\
G2  & 15160 & 15172  & 0.9992  & 14912 & G38 & 9959 & 10040  & 0.9919  & 9892 \\
G3  & 15167 & 15173  & 0.9996  & 14899 & G39 & 2837 & 2903  & 0.9773  & 2369 \\
G4  & 15176 & 15184  & 0.9995  & 14889 & G40 & 2808 & 2870  & 0.9784  & 2357 \\
G5  & 15187 & 15193  & 0.9996  & 14827 & G41 & 2803 & 2887  & 0.9709  & 2343 \\
G6  & 2631 & 2632  & 0.9996  & 2323 & G42 & 2897 & 2980  & 0.9721  & 2445 \\
G7  & 2401 & 2409  & 0.9967  & 2078 & G43 & 8571 & 8573  & 0.9998  & 8303 \\
G8  & 2420 & 2428  & 0.9967  & 2104 & G44 & 8543 & 8571  & 0.9967  & 8286 \\
G9  & 2464 & 2478  & 0.9944  & 2138 & G45 & 8543 & 8566  & 0.9973  & 8273 \\
G10 & 2401 & 2407  & 0.9975  & 2047 & G46 & 8545 & 8568  & 0.9973  & 8279 \\
G11 & 653 & 669  & 0.9761  & 669 & G47 & 8548 & 8572  & 0.9972  & 8260 \\
G12 & 645 & 660  & 0.9773  & 661 & G48 & 6000 & 6000  & 1.0000  & 6000 \\
G13 & 670 & 686  & 0.9767  & 684 & G49 & 6000 & 6000  & 1.0000  & 6000 \\
G14 & 3979 & 4012  & 0.9918  & 3939 & G50 & 6000 & 6000  & 1.0000  & 6000 \\
G15 & 3943 & 3984  & 0.9897  & 3930 & G51 & 4992 & 5037  & 0.9911  & 4958 \\
G16 & 3951 & 3991  & 0.9900  & 3924 & G52 & 4996 & 5040  & 0.9913  & 4952 \\
G17 & 3936 & 3983  & 0.9882  & 3924 & G53 & 5004 & 5039  & 0.9931  & 4958 \\
G18 & 1176 & 1207  & 0.9743  & 1027 & G54 & 4993 & 5036  & 0.9915  & 4952 \\
G19 & 1050 & 1081  & 0.9713  & 872 & G55 & 12329 & 12429  & 0.9920  & 12047 \\
G20 & 1097 & 1122  & 0.9777  & 931 & G56 & 4610 & 4752  & 0.9701  & 3850 \\
G21 & 1089 & 1109  & 0.9820  & 929 & G57 & 3985 & 4083  & 0.9760  & 4101 \\
G22 & 17080 & 17167  & 0.9949  & 16577 & G58 & 24977 & 25195  & 0.9913  & 24016 \\
G23 & 17130 & 17168  & 0.9978  & 16587 & G59 & 7097 & 7262  & 0.9773  & 5921 \\
G24 & 17117 & 17162  & 0.9974  & 16604 & G60 & 16946 & 17076  & 0.9924  & 16726 \\
G25 & 17125 & 17163  & 0.9978  & 16631 & G61 & 6668 & 6853  & 0.9730  & 5641 \\
G26 & 17103 & 17154  & 0.9970  & 16600 & G62 & 5566 & 5685  & 0.9791  & 5717 \\
G27 & 3963 & 4020  & 0.9858  & 3338 & G63 & 35001 & 35322  & 0.9909  & 33661 \\
G28 & 3921 & 3973  & 0.9869  & 3325 & G64 & 10217 & 10443  & 0.9784  & 8587 \\
G29 & 4036 & 4106  & 0.9830  & 3453 & G65 & 6341 & 6490  & 0.9770  & 6539 \\
G30 & 4063 & 4119  & 0.9864  & 3489 & G66 & 7241 & 7416  & 0.9764  & 7474 \\
G31 & 3963 & 4003  & 0.9900  & 3270 & G67 & 7904 & 8086  & 0.9775  & 8148 \\
G32 & 1618 & 1653  & 0.9788  & 1653 & G70 & 9999 & 9999  & 1.0000  & 9999 \\
G33 & 1583 & 1625  & 0.9742  & 1627 & G72 & 8010 & 8192  & 0.9778  & 8256 \\
G34 & 1573 & 1607  & 0.9788  & 1609 & G77 & 11332 & 11578  & 0.9788  & 11666 \\
G35 & 9961 & 10046  & 0.9915  & 9876 & G81 & 15985 & 16321  & 0.9794  & 16290 \\
G36 & 9954 & 10039  & 0.9915  & 9911 &     &     &     &     &  \\

\bottomrule
\end{tabular}

\end{table}

\begin{table}[p]
\footnotesize
\centering
\caption{\small Numerical cut values for MAX-$4$-CUT.}
\label{tab:MAX4CUT_res}
\begin{tabular}{l|N{2.5em}|N{2.5em}|N{2.5em}|N{3em}||l|N{2.5em}|N{2.5em}|N{2.5em}|N{3em}}
\toprule
\multicolumn{1}{l|}{Graph} & \multicolumn{1}{l|}{ODE}   & \multicolumn{1}{l|}{MOH} & \multicolumn{1}{l|}{Ratio} & \multicolumn{1}{l||}{Gurobi}
& \multicolumn{1}{l|}{Graph}  & \multicolumn{1}{l|}{ODE}  & \multicolumn{1}{l|}{MOH} & \multicolumn{1}{l|}{Ratio} & \multicolumn{1}{l}{Gurobi} \\
 \midrule

G1  & 16789 & 16803  & 0.9992  & 16443 & G37 & 11018 & 11117  & 0.9911  & 10967 \\
G2  & 16798 & 16809  & 0.9993  & 16395 & G38 & 11015 & 11108  & 0.9916  & 10970 \\
G3  & 16799 & 16806  & 0.9996  & 16375 & G39 & 2935 & 3006  & 0.9764  & 2449 \\
G4  & 16800 & 16814  & 0.9992  & 16436 & G40 & 2898 & 2976  & 0.9738  & 2487 \\
G5  & 16798 & 16816  & 0.9989  & 16413 & G41 & 2906 & 2983  & 0.9742  & 2423 \\
G6  & 2739 & 2751  & 0.9956  & 2439 & G42 & 3014 & 3092  & 0.9748  & 2563 \\
G7  & 2497 & 2515  & 0.9928  & 2114 & G43 & 9353 & 9376  & 0.9975  & 9089 \\
G8  & 2513 & 2525  & 0.9952  & 2159 & G44 & 9358 & 9379  & 0.9978  & 9086 \\
G9  & 2574 & 2585  & 0.9957  & 2120 & G45 & 9360 & 9376  & 0.9983  & 9112 \\
G10 & 2497 & 2510  & 0.9948  & 2114 & G46 & 9355 & 9378  & 0.9975  & 9100 \\
G11 & 661 & 677  & 0.9764  & 677 & G47 & 9372 & 9381  & 0.9990  & 9073 \\
G12 & 652 & 664  & 0.9819  & 665 & G48 & 6000 & 6000  & 1.0000  & 6000 \\
G13 & 680 & 690  & 0.9855  & 690 & G49 & 6000 & 6000  & 1.0000  & 6000 \\
G14 & 4396 & 4440  & 0.9901  & 4363 & G50 & 6000 & 6000  & 1.0000  & 6000 \\
G15 & 4354 & 4406  & 0.9882  & 4346 & G51 & 5518 & 5571  & 0.9905  & 5500 \\
G16 & 4375 & 4415  & 0.9909  & 4353 & G52 & 5536 & 5584  & 0.9914  & 5494 \\
G17 & 4360 & 4411  & 0.9884  & 4325 & G53 & 5528 & 5574  & 0.9917  & 5520 \\
G18 & 1207 & 1261  & 0.9572  & 1031 & G54 & 5527 & 5579  & 0.9907  & 5503 \\
G19 & 1080 & 1121  & 0.9634  & 894 & G55 & 12498 & 12498  & 1.0000  & 12498 \\
G20 & 1119 & 1168  & 0.9580  & 934 & G56 & 4722 & 4931  & 0.9576  & 4000 \\
G21 & 1133 & 1155  & 0.9810  & 923 & G57 & 4036 & 4112  & 0.9815  & 4148 \\
G22 & 18739 & 18776  & 0.9980  & 18170 & G58 & 27650 & 27885  & 0.9916  & 26567 \\
G23 & 18736 & 18777  & 0.9978  & 18155 & G59 & 7362 & 7539  & 0.9765  & 6090 \\
G24 & 18740 & 18769  & 0.9985  & 18211 & G60 & 17148 & 17148  & 1.0000  & 17148 \\
G25 & 18742 & 18775  & 0.9982  & 18135 & G61 & 6842 & 7110  & 0.9623  & 5850 \\
G26 & 18734 & 18767  & 0.9982  & 18101 & G62 & 5642 & 5743  & 0.9824  & 5793 \\
G27 & 4159 & 4201  & 0.9900  & 3433 & G63 & 38758 & 39083  & 0.9917  & 37386 \\
G28 & 4098 & 4150  & 0.9875  & 3439 & G64 & 10605 & 10814  & 0.9807  & 8884 \\
G29 & 4245 & 4293  & 0.9888  & 3474 & G65 & 6430 & 6534  & 0.9841  & 6608 \\
G30 & 4261 & 4305  & 0.9898  & 3548 & G66 & 7352 & 7474  & 0.9837  & 7553 \\
G31 & 4131 & 4171  & 0.9904  & 3376 & G67 & 8026 & 8155  & 0.9842  & 8256 \\
G32 & 1641 & 1669  & 0.9832  & 1679 & G70 & 9999 & 9999  & 1.0000  & 9999 \\
G33 & 1602 & 1638  & 0.9780  & 1644 & G72 & 8123 & 8264  & 0.9829  & 8358 \\
G34 & 1595 & 1616  & 0.9870  & 1623 & G77 & 11493 & 11674  & 0.9845  & 11827 \\
G35 & 11017 & 11111  & 0.9915  & 10957 & G81 & 16208 & 16470  & 0.9841  & 16670 \\
G36 & 11023 & 11108  & 0.9923  & 10955 &     &     &     &     &  \\

\bottomrule
\end{tabular}
\end{table}

\begin{table}[p]
\footnotesize
\centering
\caption{\small Numerical cut values for MAX-$5$-CUT.}
\label{tab:MAX5CUT_res}
\begin{tabular}{l|N{2.5em}|N{2.5em}|N{2.5em}|N{3em}||l|N{2.5em}|N{2.5em}|N{2.5em}|N{3em}}
\toprule
\multicolumn{1}{l|}{Graph} & \multicolumn{1}{l|}{ODE}   & \multicolumn{1}{l|}{MOH} & \multicolumn{1}{l|}{Ratio} & \multicolumn{1}{l||}{Gurobi}
& \multicolumn{1}{l|}{Graph}  & \multicolumn{1}{l|}{ODE}  & \multicolumn{1}{l|}{MOH} & \multicolumn{1}{l|}{Ratio} & \multicolumn{1}{l}{Gurobi} \\
 \midrule

G1  & 17695 & 17703  & 0.9995  & 17326 & G37 & 11562 & 11603  & 0.9965  & 11354 \\
G2  & 17693 & 17706  & 0.9993  & 17364 & G38 & 11552 & 11601  & 0.9958  & 11493 \\
G3  & 17691 & 17701  & 0.9994  & 17390 & G39 & 2944 & 3022  & 0.9742  & 2479 \\
G4  & 17687 & 17709  & 0.9988  & 17341 & G40 & 2884 & 2986  & 0.9658  & 2384 \\
G5  & 17697 & 17710  & 0.9993  & 17326 & G41 & 2925 & 2986  & 0.9796  & 2475 \\
G6  & 2763 & 2781  & 0.9935  & 2317 & G42 & 3037 & 3109  & 0.9768  & 2574 \\
G7  & 2495 & 2533  & 0.9850  & 2183 & G43 & 9747 & 9770  & 0.9976  & 9487 \\
G8  & 2516 & 2535  & 0.9925  & 2121 & G44 & 9742 & 9772  & 0.9969  & 9509 \\
G9  & 2574 & 2601  & 0.9896  & 2226 & G45 & 9747 & 9771  & 0.9975  & 9471 \\
G10 & 2503 & 2526  & 0.9909  & 2147 & G46 & 9742 & 9774  & 0.9967  & 9481 \\
G11 & 667 & 677  & 0.9852  & 677 & G47 & 9745 & 9775  & 0.9969  & 9471 \\
G12 & 654 & 662  & 0.9879  & 665 & G48 & 6000 & 6000  & 1.0000  & 6000 \\
G13 & 680 & 689  & 0.9869  & 690 & G49 & 6000 & 6000  & 1.0000  & 6000 \\
G14 & 4611 & 4639  & 0.9940  & 4568 & G50 & 6000 & 6000  & 1.0000  & 6000 \\
G15 & 4580 & 4606  & 0.9944  & 4548 & G51 & 5802 & 5826  & 0.9959  & 5753 \\
G16 & 4588 & 4613  & 0.9946  & 4546 & G52 & 5802 & 5837  & 0.9940  & 5761 \\
G17 & 4575 & 4603  & 0.9939  & 4514 & G53 & 5807 & 5829  & 0.9962  & 5707 \\
G18 & 1201 & 1268  & 0.9472  & 1013 & G54 & 5806 & 5830  & 0.9959  & 5723 \\
G19 & 1083 & 1132  & 0.9567  & 918 & G55 & 12498 & 12498  & 1.0000  & 12498 \\
G20 & 1132 & 1172  & 0.9659  & 937 & G56 & 4733 & 4971  & 0.9521  & 3866 \\
G21 & 1115 & 1162  & 0.9596  & 963 & G57 & 4049 & 4111  & 0.9849  & 4148 \\
G22 & 19513 & 19553  & 0.9980  & 18989 & G58 & 29019 & 29105  & 0.9970  & 28017 \\
G23 & 19513 & 19558  & 0.9977  & 18883 & G59 & 7443 & 7566  & 0.9837  & 6191 \\
G24 & 19507 & 19555  & 0.9975  & 19063 & G60 & 17148 & 17148  & 1.0000  & 17148 \\
G25 & 19506 & 19554  & 0.9975  & 18989 & G61 & 6866 & 7188  & 0.9552  & 5946 \\
G26 & 19504 & 19552  & 0.9975  & 18949 & G62 & 5650 & 5744  & 0.9836  & 5793 \\
G27 & 4178 & 4236  & 0.9863  & 3449 & G63 & 40692 & 40786  & 0.9977  & 39540 \\
G28 & 4117 & 4182  & 0.9845  & 3429 & G64 & 10732 & 10896  & 0.9849  & 8809 \\
G29 & 4255 & 4327  & 0.9834  & 3548 & G65 & 6428 & 6540  & 0.9829  & 6608 \\
G30 & 4268 & 4340  & 0.9834  & 3563 & G66 & 7364 & 7476  & 0.9850  & 7553 \\
G31 & 4157 & 4211  & 0.9872  & 3461 & G67 & 8035 & 8165  & 0.9841  & 8256 \\
G32 & 1644 & 1670  & 0.9844  & 1679 & G70 & 9999 & 9999  & 1.0000  & 9999 \\
G33 & 1612 & 1638  & 0.9841  & 1644 & G72 & 8145 & 8266  & 0.9854  & 8358 \\
G34 & 1589 & 1615  & 0.9839  & 1623 & G77 & 11516 & 11687  & 0.9854  & 11827 \\
G35 & 11547 & 11605  & 0.9950  & 11477 & G81 & 16227 & 16501  & 0.9834  & 16670 \\
G36 & 11543 & 11601  & 0.9950  & 11416 &     &     &     &     &  \\

\bottomrule
\end{tabular}
\end{table}

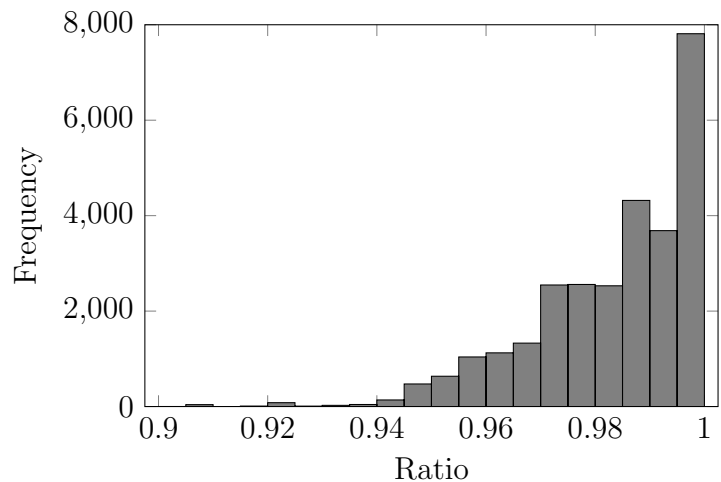
\begin{figure}[htbp]
\centering

%\definecolor{mycolor1}{rgb}{0.00000,0.44700,0.74100}%
%
\begin{tikzpicture}

\begin{axis}[%
width=3in,
height=2in,
scale only axis,
xmin=0.8975,
xmax=1.0025,
ymin=0,
ymax=8000,
xlabel={Ratio},
ylabel={Frequency},
ylabel near ticks,
axis background/.style={fill=white},
legend style={legend cell align=left, align=left, draw=white!15!black}
]
\addplot[ybar interval, fill=gray, fill opacity=0.6, draw=black, area legend] table[row sep=crcr] {%
x	y\\
0.9	1\\
0.905	40\\
0.91	0\\
0.915	10\\
0.92	81\\
0.925	9\\
0.93	25\\
0.935	43\\
0.94	138\\
0.945	473\\
0.95	635\\
0.955	1039\\
0.96	1125\\
0.965	1329\\
0.97	2547\\
0.975	2558\\
0.98	2530\\
0.985	4319\\
0.99	3686\\
0.995	7812\\
1	7812\\
};

\end{axis}
\end{tikzpicture}%
\caption{\small Quality check: Histogram of the ratios \adda{between the cut values of the ODE approach and the reference values produced by MOH} for all $4\times 71 \times 100$ trials of MAX-$k$-CUT.
}
\label{fig:MAXkCUT_qlt}
\end{figure}

\begin{table}
\footnotesize
\centering
\caption{\small Approximate total steps and time in seconds on G-Set used by running the ODE approach once for MAX-$k$-CUT  from a given initial data. }
\label{tab:MAXkCUT_time}
\begin{tabular}{c*{5}{|M}}
\toprule
   $k$ & \multicolumn{1}{l|}{$|V|$} & \multicolumn{1}{l|}{$|E|$} & \multicolumn{1}{l|}{\makecell{Number\\  of steps}} & \multicolumn{1}{l|}{Time} & \multicolumn{1}{l}{\makecell{Time\\ per step}}  \\
    \midrule
    \multirow{4}[0]{*}{2} & 8\times 10^{2} & 2\times 10^{4} & 3.0\times 10^{4} & 1.3 & 4.3\times 10^{-5} \\
& 2\times 10^{3} & 2\times 10^{4} & 1.0\times 10^{5} & 4.0 & 4.0\times 10^{-5} \\
& 1\times 10^{4} & 2\times 10^{4} & 9.0\times 10^{4} & 15 & 1.7\times 10^{-4} \\
& 2\times 10^{4} & 4\times 10^{4} & 1.5\times 10^{5} & 30 & 2.0\times 10^{-4} \\
    \midrule
\multirow{4}[0]{*}{3} & 8\times 10^{2} & 2\times 10^{4} & 9.1\times 10^{4} & 7.8   & 8.6\times 10^{-5} \\
    & 2\times 10^{3} & 2\times 10^{4} & 1.7\times 10^{5} & 28  & 1.6\times 10^{-4} \\
    & 1\times 10^{4} & 2\times 10^{4} & 1.1\times 10^{5} & 55  & 5.0\times 10^{-4} \\
    & 2\times 10^{4} & 4\times 10^{4} & 1.2\times 10^{5} & 84  & 7.0\times 10^{-4} \\
    \midrule
\multirow{4}[0]{*}{4} & 8\times 10^{2} & 2\times 10^{4} & 1.2\times 10^{5} & 20  & 1.5\times 10^{-4} \\
    & 2\times 10^{3} & 2\times 10^{4} & 1.4\times 10^{5} & 28  & 2.0\times 10^{-4} \\
    & 1\times 10^{4} & 2\times 10^{4} & 1.0\times 10^{6} & 60  & 6.0\times 10^{-4} \\
    & 2\times 10^{4} & 4\times 10^{4} & 1.2\times 10^{5} & 90  & 7.5\times 10^{-4} \\
    \midrule
\multirow{4}[0]{*}{5} & 8\times 10^{2} & 2\times 10^{4} & 1.1\times 10^{5} & 13   & 1.2\times 10^{-4} \\
    & 2\times 10^{3} & 2\times 10^{4} & 1.5\times 10^{5} & 30  & 2.0\times 10^{-4} \\
    & 1\times 10^{4} & 2\times 10^{4} & 1.2\times 10^{5} & 70  & 5.5\times 10^{-4} \\
    & 2\times 10^{4} & 4\times 10^{4} & 1.3\times 10^{5} & 86  & 6.6\times 10^{-4} \\
    \bottomrule
\end{tabular}%
\end{table}

\section{Application to approximating the star discrepancy}
\label{sec:discrepancy}
%%%???? an important problem ......

For a positive integer $d$, vectors $a,b\in\mathbb{R}^d$ and relation $\triangle\in \{<, >,\leq,\geq\}$, let $a\bigtriangleup b$ represent that the relation holds for every entry, namely $a_i \bigtriangleup b_i$ for all $i\in[d]$. We define a closed interval $[a,b]:=\left\{u\in \mathbb{R}^d\middle|u\geq a\; \text{and } u\leq b\right\}$ for $a<b$. The half-open and open cases can be similarly expressed. The volume for an interval denotes as $\vol([a,b]):=\prod (b_i-a_i)$. When $a=0$, we drop off $a$ and simplify it as $\vol(b):=\vol([a,b])$. To represent a series of vector with indices, we let those indices appear as superscripts of a single symbol. The subscripts are left for denoting the entries of these vectors. Given an $N$-point set $U=\{u^1,\dots,u^N\}\subset [0,1)^d$,  $u=\left(u_1,\dots,u_d\right)\in[0,1]^d$, let 
\begin{equation}\label{eq:AD}
A(u; U)=\left|[0,u)\cap U\right|, \quad D(u; U) = \vol(u)-\frac{1}{N} A(u; U). 
\end{equation}
The star discrepancy of the $N$-point set $U$ is defined by
\begin{equation}\label{eq:star0}
d_{\infty}^*(U)=\sup_{u\in[0,1]^d} |D(u; U)|,
\end{equation}
which measures the uniformity of $U$,
and has been widely used in high dimensional numerical integration \cite{koksma_general_1942,hlawka_funktionen_1961} as well as in high dimensional statistics such as number-theoretical methods \cite{fang_number-theoretic_1993} and density estimation \cite{li_density_2016,shao_spade_2020}. 
However, calculating $d_{\infty}^*(U)$ admits NP-hardness \cite{gnewuch_finding_2009} and only a few algorithms were developed. They are the exact algorithm \cite{bundschuh_method_1993},
the threshold-accepting algorithm \cite{winker_application_1997} and its variation, dubbed the TA\_improved algorithm \cite{gnewuch_new_2011}. In this section, we will apply the proposed ODE approach for MCPP to approximate $d_{\infty}^*(U)$. To the best of our knowledge, this is the first attempt to use a continuous algorithm for approximating the star discrepancy.

%\[
%\Gamma_j(U)=\{u_j^i|i=1,2,\dots,N\},\quad \bar{\Gamma}_j(U)=\Gamma_j(U)\cup\{1\},\quad \underline{\Gamma}_j(U)=\Gamma_j(U)\cup\{0\}, j=1,2,\dots,d
%\]
%and
%\[
%\Gamma(U)=\Gamma_{1}(U) \times \cdots \times \Gamma_{d}(U).
%\]
%$\bar\Gamma(U)$ and $\underline\Gamma(U)$ are defined similarly.

For $j\in[d]$, we define $\Gamma_j(U)=\{u_j^i\, \big|\, i=1,2,\dots,N\}$, $\bar{\Gamma}_j(U)=\Gamma_j(U)\cup\{1\}$, 
$ \underline{\Gamma}_j(U)=\Gamma_j(U)\cup\{0\}$, and $\Gamma(U)=\Gamma_{1}(U) \times \cdots \times \Gamma_{d}(U)$ ($\bar\Gamma(U)$ and $\underline\Gamma(U)$ can be defined in a similar way). 
It can be readily observed that, on each open sub-domain of $[0,1]^d$ divided by the grids in $\Gamma(U)$,
$A(u;U)$ keeps unchanged and thus $|D(u; U)|$ reaches its maximum at one of the extreme points of $\vol(u)$, i.e., either the lower left or upper right corner. Since $\underline\Gamma(U)$ and $\bar\Gamma(U)$ respectively collect all the lower left and upper right corners in $[0,1]^d$, an equivalent form for the star discrepancy in Eq.~\eqref{eq:star0} was obtained in \cite{niederreiter_discrepancy_1972}: 
\begin{equation}
\label{eq:Enm0}
d_{\infty}^*(U)=\max \left\{\max _{u \in \bar{\Gamma}(U)} D(u; U), \max _{u\in \underline\Gamma(U)} \bar{D}(u; U)\right\},
\end{equation}
where  
\begin{equation}\label{eq:ADbar}
\bar{A}(u; U)=\left|[0,u]\cap U\right|, \quad 
\bar{D}(u; U)  = \frac{1}{N} \bar{A}(u; U)-\vol(u).
\end{equation}
We should mention that,  $\underline\Gamma(U)$, the feasible region for maximizing $\bar{D}(u; U)$,  can be slightly narrowed to $\Gamma(U)$ due to the following reasons. For each $j\in [d]$, if $0\in \Gamma_j(U)$, then $\underline{\Gamma}_j(U)$ equals $\Gamma_j(U)$. Otherwise, $\bar{A}(u; U)=0$ if $u_j=0$,  thus $\bar{D}(u;U)=0<d_{\infty}^*(U)$, indicating that we can drop out that kind of $u$ when maximizing $\bar{D}(u;U)$. In either case, we can modify the Cartesian product expression of $\Gamma(U)$ by using $\Gamma_j(U)$ instead of $\underline{\Gamma}_j(U)$. Therefore, from Eq.~\eqref{eq:Enm0}, we achieve a more concise form for calculating the star discrepancy,
\begin{equation}
\label{eq:Enm}
d_{\infty}^*(U)=\max \left\{\max _{u \in \bar{\Gamma}(U)} D(u; U), \max _{u\in \Gamma(U)} \bar{D}(u; U)\right\}.
\end{equation}
The equivalent form~\eqref{eq:Enm} transforms the star discrepancy into two optimization problems on discrete sets with which we are able to obtain MCPP forms.

{Let us} start from the first optimization problem $\max _{u \in \bar{\Gamma}(U)} D(u; U)$ in Eq.~\eqref{eq:Enm}. Inspired by the Cartesian-product structure of $\bar{\Gamma}(U)$, the feasible region of corresponding MCPP instance should be $B_{N+1}^d$ since the cardinality of $\bar{\Gamma}_j$ is no larger than $N+1$.  That is, the key parameters which shape the MCPP problem in Eqs.~\eqref{eq:mcpp0} and \eqref{eq:X} are:  $n \leftarrow (N+1)d$, $m \leftarrow d$, $d_j \leftarrow N+1$ for any $j\in[m]$. However, mapping $\bar{\Gamma}(U)$ into $B_{N+1}^d$ and defining an objective function over $B_{N+1}^d$ that represents $D(u;U)$ need meticulous designs we are about to state below. % in what follows.

For $j \in [d]$, sort $\bar\Gamma_j(U)$ into $\bar{u}_{1j}\leq \bar{u}_{2j}\leq\cdots\leq \bar{u}_{Nj}<\bar{u}_{(N+1)j}=1$, which also sort $\Gamma_j(U)$ in the same order for $\Gamma_j(U) = \{\bar{u}_{1j},\dots,\bar{u}_{Nj}\}$. 
For $i\in [N+1]$, let $\bar{u}_{\sigma_{ij}j} := u_j^{i}$ record the order of $u_j^i$ and 
$\{\sigma_{ij} \, \big| \, i\in [N+1]\}$ be the corresponding permutation of $[N+1]$.  
When some elements of $\Gamma_j(U)$ are identical, such permutation may not be unique,
and it is viable to choose one of them arbitrarily. Let $x\in B_{N+1}^d$. For each $j\in [d]$, there is a unique entry of $x^{(j)}$, denoted by $x^{(j)}_{s_j}$ with $s_j\in[N+1]$,  that equals $1$. Then we have a surjection from $x\in B_{N+1}^d$ to $u\in \bar{\Gamma}(U)$
\begin{equation}
\label{map}
x^{(j)} \rightarrow s_j \rightarrow   u_j=\bar{u}_{s_j j}, \quad \forall\, j\in[d]. 
\end{equation}
Let
\begin{equation}
\label{eq:3o}
\nu(x) := \prod_{j=1}^d\sum_{i=1}^{N+1} x^{(j)}_{i} \bar{u}_{ij},\;\;\;
\alpha(x) :=\sum_{i=1}^N \prod_{j=1}^d \sum_{k=\sigma_{ij}+1}^{N+1} x^{(j)}_{k},\;\;\;
\delta(x) := \nu(x)-\frac{1}{N}\alpha(x), 
\end{equation}
all of which are affine with respect to each $x^{(j)}$. In fact, 
$\nu(x)$, $\alpha(x)$ and $\delta(x)$ act as $\vol(u)$, $A(u;U)$ and $D(u;U)$, respectively.
{For $x\in B_{N+1}^d$, it should be noticed that the inner sum can be replaced by the logical disjunction ``$\vee$" and the product can be replaced by the logical conjunction ``$\wedge$" (regarding $0$-$1$ variables as boolean variables) , namely, 
\[
\alpha(x)=\sum_{i=1}^N \bigwedge_{j=1}^d \bigvee_{k=\sigma_{ij}+1}^{N+1} x^{(j)}_{k},
\]
since $\sum_{k=\sigma_{ij}+1}^{N+1} x^{(j)}_{k}\in \{0,1\}$. 
}

%We denote it by $x^{(j)}_{s_j}=1$ , and then map $x$ to $u\in \bar{\Gamma}(U)$: $x^{(j)} \rightarrow s_j \rightarrow   \bar{u}_{s_j j}$, namely, $u_j=\bar{u}_{s_j j}$.  

% \]
%Therefore,
%\[
%\vol(u) =\prod_{j=1}^d u_j= \prod_{j=1}^d\sum_{i=1}^{N+1} x^{(j)}_{i} \bar{u}_{ij}=\nu(x).
%\]

\begin{prop}
\label{prop:poly_disc}
Let $u\in \bar\Gamma(U)$ be the image of $x\in B_{N+1}^d$ under the surjection~\eqref{map}. Then 
we have
\begin{align} 
\vol(u) &= \nu(x), \label{eq:vnu}\\
A(u;U) &\leq \alpha(x), \label{eq:Aalpha}\\
D(u;U) &\geq \delta(x). \label{eq:Ddelta}
\end{align}
Moreover, $D(u;U)$ and $\delta(x)$ admit a deeper connection, 
\begin{equation}
\max_{u\in\bar{\Gamma}(U)}D(u;U)=\max_{x\in B_{N+1}^d} \delta(x). \label{eq:max_D_delta}
\end{equation}
\end{prop}

\begin{proof}
It is easy to verify Eq.~\eqref{eq:vnu} from the fact $\sum_{i=1}^{N+1} x^{(j)}_{i} \bar{u}_{ij}=u_j$.
Next {let us} express $|[0,u)\cap \{u^i\}|$ by $x$. For each $i\in[N]$, we have

\begin{align}
                  u^i\in [0,u)
\Leftrightarrow\ & u_j^i=\bar{u}_{\sigma_{ij}j}<u_j=\bar{u}_{s_j j},\;\;\; \forall\, j\in[d],\notag\\
\Rightarrow\ & \sigma_{ij}<s_j,\;\;\; \forall\, j\in[d], \label{eq:suff}\\
\Leftrightarrow\ & \bigvee_{k=\sigma_{ij}+1}^{N+1} x^{(j)}_{k}=1,\;\;\; \forall\, j\in[d].\notag
\end{align}
Thus $|[0,u)\cap \{u^i\}|\leq \bigwedge_{j=1}^d \bigvee_{k=\sigma_{ij}+1}^{N+1} x^{(j)}_{k}$ and summing up for all $i\in [N]$ leads to Eq.~\eqref{eq:Aalpha}.  Combining Eqs.~\eqref{eq:AD}, \eqref{eq:3o}, \eqref{eq:vnu} and \eqref{eq:Aalpha} together, we immediately arrive at Eq.~\eqref{eq:Ddelta},
and thus $\max_{u\in \bar{\Gamma}(U)} D(u;U)\geq \max_{x\in B_{N+1}^d}\delta(x)$ for 
the mapping~\eqref{map} from $B_{N+1}^d$ to $\bar{\Gamma}(U)$ is surjective. 

To prove Eq.~\eqref{eq:max_D_delta}, we only have to show $\max_{u\in \bar{\Gamma}(U)} D(u;U)\leq \max_{x\in B_{N+1}^d}\delta(x)$, which can be obtained via the fact that,  for each $u\in \bar{\Gamma}(U)$, there exists a preimage of $u$, $\hat{x}=\hat{x}(u)\in B_{N+1}^d$, such that $A(u;U)=\alpha(\hat{x}(u))$. The $\hat{x}$ is defined by letting $s_j$ equal the minimal index that $u_j=\bar{u}_{s_j j}$ holds and $\hat{x}_{s_j}^{(j)}=1,\ j\in [d]$. Thus $\bar{u}_{(s_j-1) j}<\bar{u}_{s_j j}$ whenever $s_j>1$, otherwise contradicting to the minimality of $s_j$. When $\sigma_{ij}<s_j$, we then have
\begin{equation}
\sigma_{ij}\leq s_j-1\Rightarrow \bar{u}_{\sigma_{ij}j}\leq \bar{u}_{(s_j-1) j} < \bar{u}_{s_j j}.
\end{equation}
Therefore, the ``$\Rightarrow$" in Eq.~\eqref{eq:suff} becomes ``$\Leftrightarrow$", which results in the equivalence between $|[0,u)\cap \{u^i\}|$ and $\bigwedge_{j=1}^d \bigvee_{k=\sigma_{ij}+1}^{N+1} \hat{x}^{(j)}_{k}$, rather than {an} inequality. Hence, we get $A(u;U)=\alpha(\hat{x}(u))$. Let $u^*$ reaches the maximum of $D(u;U)$. Then 
\[
\max_{u\in \bar{\Gamma}(U)} D(u;U) = D(u^*;U)=\delta(\hat{x}(u^*))\leq \max_{x\in B_{N+1}^d}\delta(x).
\]
%The proof is completed. 
\end{proof}
%
%Therefore, for all pairs $(x, u)$ under under the surjection~\eqref{map} from $B_{N+1}^d$ to $\bar{\Gamma}(U)$,
%we have $D(u;U)\geq \delta(x)$ from Proposition~\ref{prop:poly_disc}. 
%
%
%for all pairs $(u,x)$ satisfying the conditions in Prop.~\ref{prop:poly_disc}. However, for each $u\in \bar{\Gamma}(U)$, there exists a $x\in B_{N+1}^d$ such that $D(u;U)=\delta(x)$ by letting $s_j$ equals the minimal index that $u_j=\bar{u}_{s_j j}$ holds. This can be seen by noticing that the 
%``$\Rightarrow$" in Eq.~\eqref{eq:suff} becomes ``$\Leftrightarrow$" if we let $s_j$ satisfy $\bar{u}_{(s_j-1) j}<\bar{u}_{s_j j}$. In other words, $\max_{x\in B_{N+1}^d} \delta(x)=\max_{u\in\bar{\Gamma}(U)}D(u;U)$, so solving the first MCPP instance leads to the solution for maximizing $D(u;U)$.

Eq.~\eqref{eq:max_D_delta} gives an MCPP form for the first optimization problem in Eq.~\eqref{eq:Enm}
and the same treatment can be also applied to the second one. The key parameters which shape the MCPP problem in Eqs.~\eqref{eq:mcpp0} and \eqref{eq:X} become:  $n \leftarrow Nd$, $m \leftarrow d$, $d_j \leftarrow N$ for any $j\in[m]$, and we need the following functions defined on $B_N^d$: 
\begin{equation*}
\bar{\nu}(x') := \prod_{j=1}^d\sum_{i=1}^{N} (x')^{(j)}_{i} \bar{u}_{ij},\;\;\;
\bar{\alpha}(x') := \sum_{i=1}^N \prod_{j=1}^d \sum_{k=\sigma_{ij}}^{N} (x')^{(j)}_{k},\;\;\;
\bar{\delta}(x') := \frac{1}{N}\bar{\alpha}(x')-\bar{\nu}(x'),
\end{equation*}
to present the MCPP form
\begin{equation}\label{eq:max_barDdelta}
\max_{u\in\Gamma(U)}\bar{D}(u;U)=\max_{x'\in B_N^d} \bar{\delta}(x'),
\end{equation}
the proof of which is very similar to that of Eq.~\eqref{eq:max_D_delta} and thus skipped here. In a word, according to Eqs.~\eqref{eq:max_D_delta} and \eqref{eq:max_barDdelta}, the star discrepancy is now determined by two MCPP problems
\begin{equation}\label{eq:MCPP_disc}
d_{\infty}^*(U)=\max \left\{\max_{x\in B_{N+1}^d}\delta(x), \max_{x'\in B_N^d}\bar{\delta}(x')\right\},
\end{equation}
which can be solved approximately by the proposed ODE approach in a straightforward manner. % (?????order???) % 目标函数的次数

Compared to the quadratic objective function of MAX-$k$-CUT problem in Eq.~\eqref{eq:MkC_MCPP}, both $\delta(x)$ and $\bar{\delta}(x')$ in Eq.~\eqref{eq:MCPP_disc} are of degree of $d$ and it may be more complicated to store and calculate these function values as well as their gradients when $d>2$. 
For example, calculating the gradients of $\delta$ and $\bar{\delta}$ may involve high computational cost if handling it inappropriately. It forms the main cost in using FE scheme~\eqref{fe} to solve Eq.~\eqref{eq:MCPP_disc}, thereby requiring careful optimization.
More precisely, it is easy to check that for $i\in [N+1],j\in[d]$,
\begin{equation}\label{eq:gg}
\frac{\partial \nu(x)}{\partial x^{(j)}_i} =\bar{u}_{ij}\prod_{j'\in[d]\backslash \{j\}}\sum_{i'=1}^{N+1}x^{(j')}_{i'}\bar{u}_{i'j'}, \;\;\;
\frac{\partial \alpha(x)}{\partial x^{(j)}_i} =\sum_{\substack{i'\in [N]\\ \sigma_{i'j}<i}} \prod_{j'\in[d]\backslash \{j\}}\sum_{k=\sigma_{i'j'}+1}^{N+1}x^{(j')}_k.
\end{equation}
If we calculate the sums and products in $\frac{\partial \alpha(x)}{\partial x^{(j)}_i}$ straightforwardly, there will be at least $\mathcal{O}(N^2d)$ operations. Considering that there are $(N+1)d$ entries, the complexity of $\nabla \delta(x)$ will be at least $\mathcal{O}(N^3d^2)$, which is barely acceptable. However, we would like to point out that the complexity can be limited to $\mathcal{O}(Nd)$ according to the following procedure.  First, it should be observed that given $U$, Eq.~\eqref{eq:gg} can be simplified as
\begin{equation}
\frac{\partial \nu(x)}{\partial x^{(j)}_i}=\bar{u}_{ij}\prod_{j'\in[d]\backslash \{j\}} b_{j'}(x),\;\;\; \frac{\partial \alpha(x)}{\partial x^{(j)}_i}=\sum_{\substack{i'\in [N]\\ \sigma_{i'j}<i}} \prod_{j'\in[d]\backslash \{j\}}z_{i'j'}(x),
\end{equation}
where the vector function $b:[0,1]^n\rightarrow \mathbb{R}^{d}$ and the matrix function $z:[0,1]^n\rightarrow \mathbb{R}^{N\times d}$ are independent with $i$ or $j$. Thus $b$ and $z$ can be calculated in advance. To calculate $z$ efficiently, noting that $z_{i'j'}(x)=\sum_{k>\sigma_{i'j'}}x^{(j')}_k$ is a partial sum of entries of $x^{(j')}$ in reverse order, we can compute each column of $z$ in a group within a cost of only $\mathcal{O}(N)$, thereby a cost of $\mathcal{O}(Nd)$ for computing $z$. To get $g_{i'j}(x):=\prod_{j'\in[d]\backslash \{j\}}z_{i'j'}(x)$ fast, we replace the multiplication by $\prod_{j'\in[d]}z_{i'j'}(x)/z_{i'j}(x)$, whose numerator can be reused for fixed $i'$. Therefore, the complexity of computing $g$ from $z$ is also $\mathcal{O}(Nd)$. 
To get $\nabla \alpha$ from $g$, the partial sum trick can also be applied in the summation $\sum_{i':\sigma_{i'j}<i}g_{i'j}(x)$,  albeit $\sigma_{\cdot j}$ determining the order of entries in $g_{\cdot j}$. Hence, the total time cost is $\mathcal{O}(Nd)$.  In the same spirit, we are able to get $\nabla \nu$ with cost of $\mathcal{O}(Nd)$.   
By these means, the complexity of $\nabla\delta$ is limited to $\mathcal{O}(Nd)$. This is also true for $\nabla\bar{\delta}$. Since $n=\mathcal{O}(Nd)$ for these two MCPP instances, the cost of a FE step is also $\mathcal{O}(Nd)$ by the observation in Section~\ref{subsec:cost}.

 % Since $\sigma_{\cdot j}$ is a permutation of $[N+1]$ when $j$ is fixed, the summation is partial sum of entries of $g_{\cdot j}$ in an order determined by $\sigma_{\cdot j}$. Another trick is replacing $\prod_{j'\in [d]\backslash \{j\}}b_{j'}$ by $\prod_{j'\in[d]} b_{j'}\big/ b_j$. The numerator can be reused, therefore the complexity of getting $\prod_{j'\in [d]\backslash \{j\}}b_{j'},\ j\in [d]$ is also $\mathcal{O}(d)$ when $b_{j'}$ is known. This reusing also works when computing $g_{i'j}$.
 %there are two observations from Eq.~\eqref{eq:gg}. (i) The product in the form of $\prod_{j'\neq j}$ can be computed by dividing by each term from the product of all terms; (ii)are partial sums of some $N$ terms when $i$ varies from $2$ to $N+1$ since $\sigma_{\cdot j}$ is a permutation of $[N+1]$ when $j$ is fixed. The same property holds for $\sum_{k=\sigma_{i'j'}+1}^{N+1}$ as well. Therefore, we do not have to explicitly compute the summation for every $i,j$, and instead we sequentially compute the partial sum. (????re-use??) These observations hold for $\bar{\nu}$ and $\bar{\alpha}$ as well. With the help of them, the time complexity for computing the gradient of $\delta$ and $\bar{\delta}$ can be reduced to $\mathcal{O}(Nd)$, which is proportional to size of $x$ or $x'$. Since $n=\mathcal{O}(Nd)$ for these two MCPP instances, the cost of a FE step is also $\mathcal{O}(Nd)$ by the observation in Section~\ref{subsec:cost}.

The GLP sets used in \cite{winker_application_1997} are adopted for test in this work,
and include $30$ small sets with $N$ ranging from $28$ to $487$ and $d$ from $4$ to $6$
as well as $6$ large sets with $N$ from around $2000$ to $5000$ and $d$ from $6$ to $11$. 
The parameters for our ODE approach are $T_1=1\times 10^{-4}$, $T_s=\gamma^{s-1}T_1,s=2,3,\dots$, $\gamma = 0.95$, $\varepsilon_0=1\times 10^{-3}$, $\Theta = 1\times 10^{-6}Nd$ and $\rho = 1.1$. {The parameters are similar to those in MAX-$k$-CUT (see Section~\ref{sec:MAXkCUT}.)}
Numerical values of the star discrepancy obtained by re-running the ODE approach 100 times are listed in Table~\ref{tab:disc_res}, where those obtained by TA\_improved are directly copied from \cite{gnewuch_new_2011}. 

We list the ratios \adda{between the best discrepancy} values achieved by our ODE method and \adda{by} TA\_improved in Table~\ref{tab:disc_res}. It can be seen that except for $4$ out of $30$ small instances, the ODE method gives the same value as TA\_improved did, see $(N,d)=(312,4), (376,4), (487,4), (73,6)$ in Table~\ref{tab:disc_res}, and for these four instances, the ratios are at least $0.98$. For the last $6$ large instances, the ratios are no smaller than $0.91$. 

The histogram of the ratios \adda{between the discrepancy values of the ODE approach and the reference values obtained from TA\_improved} for $36\times 100$ trials is plotted in Figure~\ref{fig:disc_qlt}. In more than $85\%$ trials, the ratios are greater than $0.9$. Compared to TA\_improved's 100000 iterations per trial used in \cite{gnewuch_new_2011} and the same $\mathcal{O}(Nd)$ complexity per iteration, our method requires much less computational resource overall while producing solutions with similar quality. We also record the approximate typical runtime and steps for different sizes of sets in Table~\ref{tab:disc_time}. It can be seen there that the time consumed per iteration is roughly proportional to $Nd$. 

%Gurobi is adopted as reference solver of MCPP as before. 

Since Gurobi 10.0.1 can only deal with quadratic objective functions and constraints, we need to reformulate Eq.~\eqref{eq:MCPP_disc}. For the $\nu(x)$ part of $\delta(x)$, the procedure is straightforward. We create $d$ variables $\mu_j,j=1,2,\dots,d$ and constrain them by $\mu_j=\sum_{i=1}^{N+1} x^{(j)}_{i} \bar{u}_{ij}$. Then the product of $\mu_j$ is modeled in a standard way as follows. We create $\tau_j,\,j=2,3,\dots,d$, satisfying $\tau_j=\mu_1\mu_2\cdots\mu_j$, which can be done by quadratic constraints $\tau_2=\mu_1\mu_2$, $\tau_j=\tau_{j-1}\mu_j,\,j=3,4,\dots,d$. For the $\alpha(x)$ part, we need $\mathcal{O}(Nd)$ variables representing the subsum terms. The products in $\alpha(x)$ are modelled by Gurobi's standard general constraints ``AND". $\bar{\delta}(x)$ is handled in a similar way.

The values obtained by Gurobi limited to the same runtime of ODE approach are also listed in Table~\ref{tab:disc_res}. Since in some instances, Gurobi fails to produce any feasible solutions, we let it run until it finds a solution no worse than our approach's, then we record the runtime (as there are two maximum problems in one instance, we choose the one with larger objective value found by ODE approach).  
The columns headed by ``Gurobi" in Table~\ref{tab:disc_res} show the results by two means. That is, the numbers without parentheses represent the objective function values found by Gurobi under the limitation of runtime, while the ones in parentheses represent the quotients of runtimes of Gurobi and our method (total time for two problems). In the first $30$ instances, Gurobi manages to find feasible solutions, but the objective values are all less than or equal to our method's. For $(N,d)=(3997,9),(4661,10),(4661,11)$, Gurobi fails to provide a competitive solution compared to ODE method after running for more than 20 hours (therefore, we leave a dash there). 

\begin{table}[htbp]
\footnotesize
\centering
\caption{\small Numerical values of the star discrepancy on GLP sets. Values in the columns headed by ``TA" are directly copied from \cite{gnewuch_new_2011} as reference and obtained there by the TA\_improved algorithm. {Values in the columns headed by ``Gurobi" are either the objective values or the quotients of runtimes (in parentheses). 
A dash ``\text{-}" means Gurobi fails to provide a competitive solution compared to the ODE method after running for more than 20 hours. 
}} % 删除线的两部分似乎可以删掉？更加简洁
\label{tab:disc_res}
\begin{tabular}{*{2}{M|}*{3}{N{2.6em}|}M||*{2}{M|}*{3}{N{2.6em}|}M}
\toprule
\multicolumn{1}{l|}{$N$} & \multicolumn{1}{l|}{$d$} & \multicolumn{1}{l|}{ODE} & \multicolumn{1}{l|}{TA} & \multicolumn{1}{l|}{Ratio}  & \multicolumn{1}{l|}{Gurobi} &  \multicolumn{1}{l|}{$N$} & \multicolumn{1}{l|}{$d$} & \multicolumn{1}{l|}{ODE} & \multicolumn{1}{l|}{TA} & \multicolumn{1}{l|}{Ratio} & \multicolumn{1}{l}{Gurobi}\\
\midrule
145  & 4   & 0.0731  & 0.0731  & 1.0000  & 0.0731 & 28  & 6   & 0.5360  & 0.5360  & 1.0000  & 0.5360 \\
255  & 4   & 0.1093  & 0.1093  & 1.0000  & 0.1016 & 29  & 6   & 0.2532  & 0.2532  & 1.0000  & 0.2532 \\
312  & 4   & 0.0617  & 0.0618  & 0.9974  & 0.0595 & 35  & 6   & 0.3431  & 0.3431  & 1.0000  & 0.3431 \\
376  & 4   & 0.0752  & 0.0753  & 0.9979  & 0.0682 & 50  & 6   & 0.3148  & 0.3148  & 1.0000  & 0.3148 \\
388  & 4   & 0.1297  & 0.1297  & 1.0000  & 0.0989 & 61  & 6   & 0.1937  & 0.1937  & 1.0000  & 0.1937 \\
442  & 4   & 0.0620  & 0.0620  & 1.0000  & 0.0424 & 73  & 6   & 0.1467  & 0.1485  & 0.9876  & 0.1406 \\
448  & 4   & 0.0548  & 0.0548  & 1.0000  & 0.0178 & 81  & 6   & 0.2500  & 0.2500  & 1.0000  & 0.2498 \\
451  & 4   & 0.0271  & 0.0271  & 1.0000  & 0.0146 & 88  & 6   & 0.2658  & 0.2658  & 1.0000  & 0.2658 \\
471  & 4   & 0.0286  & 0.0286  & 1.0000  & 0.0225 & 90  & 6   & 0.1992  & 0.1992  & 1.0000  & 0.1605 \\
487  & 4   & 0.0413  & 0.0413  & 0.9995  & 0.0138 & 92  & 6   & 0.1635  & 0.1635  & 1.0000  & 0.1631 \\
102  & 5   & 0.1216  & 0.1216  & 1.0000  & 0.1160 & 2129  & 6   & 0.0246  & 0.0254  & 0.9685  & (1.1) \\
122  & 5   & 0.0860  & 0.0860  & 1.0000  & 0.0791 & 3997  & 7   & 0.0251  & 0.0254  & 0.9882  & (8.3) \\
147  & 5   & 0.1456  & 0.1456  & 1.0000  & 0.1107 & 3997  & 8   & 0.0242  & 0.0254  & 0.9528  & (29) \\
153  & 5   & 0.1075  & 0.1075  & 1.0000  & 0.0871 & 3997  & 9   & 0.0387  & 0.0387  & 1.0000  & \text{-} \\
169  & 5   & 0.0755  & 0.0755  & 1.0000  & 0.0710 & 4661  & 10  & 0.0256  & 0.0272  & 0.9412  & \text{-} \\
170  & 5   & 0.0860  & 0.0860  & 1.0000  & 0.0771 & 4661  & 11  & 0.0259  & 0.0283  & 0.9152  & \text{-} \\
195  & 5   & 0.1574  & 0.1574  & 1.0000  & 0.1193 &     &     &     &     &     &  \\
203  & 5   & 0.1675  & 0.1675  & 1.0000  & 0.1260 &     &     &     &     &     &  \\
235  & 5   & 0.0786  & 0.0786  & 1.0000  & 0.0606 &     &     &     &     &     &  \\
236  & 5   & 0.0582  & 0.0582  & 1.0000  & 0.0466 &     &     &     &     &     &  \\

\bottomrule
\end{tabular}%
\end{table}

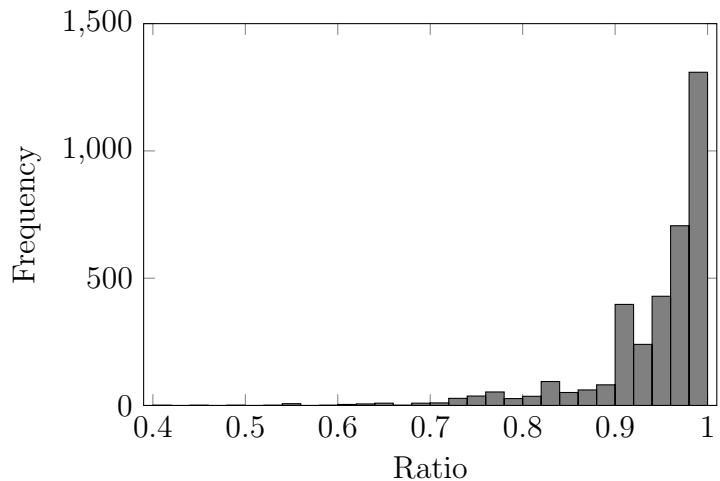
\begin{figure}[htpb]
\centering
\begin{tikzpicture}
\begin{axis}[%
width=3in,
height=2in,
scale only axis,
xmin=0.39,
xmax=1.01,
ymin=0,
ymax=1500,
xlabel={Ratio},
ylabel={Frequency},
ylabel near ticks,
axis background/.style={fill=white},
]
\addplot[ybar interval, fill=gray, fill opacity=0.6, draw=black, area legend] table[row sep=crcr] {%
x	y\\
0.4	1\\
0.42	0\\
0.44	1\\
0.46	0\\
0.48	1\\
0.5	0\\
0.52	1\\
0.54	7\\
0.56	0\\
0.58	1\\
0.6	4\\
0.62	6\\
0.64	9\\
0.66	1\\
0.68	9\\
0.7	10\\
0.72	28\\
0.74	37\\
0.76	53\\
0.78	27\\
0.8	36\\
0.82	94\\
0.84	51\\
0.86	61\\
0.88	81\\
0.9	397\\
0.92	240\\
0.94	429\\
0.96	706\\
0.98	1309\\
1	1309\\
};

\end{axis}
\end{tikzpicture}
\caption{\small Quality check: Histogram of the ratios \adda{between the discrepancy values of the ODE approach and the reference values produced by TA\_improved} for all $36 \times 100$ trials.}
\label{fig:disc_qlt}
\end{figure}

\begin{table}[htbp]
\centering
\footnotesize
\caption{\small Approximate total steps and time in milliseconds used by the ODE approach for approximating the star discrepancy in Eq.~\eqref{eq:MCPP_disc}.}
\label{tab:disc_time}
\begin{tabular}{*{4}{M|}M}
\toprule
\multicolumn{1}{l|}{$N$} & \multicolumn{1}{l|}{$d$} & \multicolumn{1}{l|}{Time} & \multicolumn{1}{l|}{\makecell{Number\\ of steps}} & \multicolumn{1}{l}{\makecell{Time\\ per step}} \\
\midrule
1.0\times 10^2 & 4 & 4.0\times 10^1 & 8.0\times 10^2 & 5.0\times 10^{-2}  \\
2.5\times 10^2 & 4 & 6.0\times 10^1 & 8.0\times 10^2 & 6.0\times 10^{-2}   \\
5.0\times 10^2 & 4 & 1.0\times 10^2 & 1.5\times 10^3 & 7.0\times 10^{-2}   \\
5.0\times 10^2 & 6 & 6.0\times 10^2 & 5.0\times 10^3 & 1.2\times 10^{-1}   \\
2.0\times 10^3 & 6 & 3.2\times 10^3 & 8.0\times 10^3 & 4.0\times 10^{-1}   \\
4.0\times 10^3 & 10& 1.3\times 10^4 & 1.3\times 10^4 & 1.0\times 10^{0}   \\
\bottomrule
\end{tabular}%
\end{table}

%\section{Conclusion}
\section{Conclusion and discussion}
\label{sec:conclu}

%{a continuous-time Markov chain defined on the feasible region}.

We proposed an ODE approach for multiple choice polynomial programming (MCPP) and demonstrated its validity via both theoretical analysis and numerical experiments. It fully exploits a connection between the discrete MCPP problem and the continuous ODE system through {revealing the relation between local optima of the MCPP and equilibriums of the ODE. The resulting solutions of MCPP instances representing two specific problems are relatively good compared to dedicated algorithms', and are mostly competitive compared to Gurobi's.
We are conducting analysis on the existence of equilibrium points and the realizability of conditions that ensures the local optimality, and trying more advanced numerical techniques for evolving an ODE to its equilibrium points. We are going to extend the proposed ODE approach to some kinds of mixed integer programming problems with more constraints rather than multiple choice. On the other hand, although the polytope of unconstrained pseudo-boolean optimization has been thoroughly studied \cite{del_pia_polyhedral_2017}, there is very limited research on the polytope of MCPP, and thus accelerating the ODE approach with the aid of polyhedral property and/or cutting-plane method is also a subject of future research. Moreover, we hope that our preliminary attempt in this work may inspire more new connections between discrete data world and continuous math field.

%by numerically integrating Eq.~\eqref{eq:main_ODE}
%
%Eq.~\eqref{eq:main_ODE}
%
%The future research may be done both theoretically and practically. The main problems left for theoretical analysis are the existence of the equilibrium and the realizability of the condition of the local minimality. In other words, the problem proposed in Remark \ref{rmk:cond} requires further investigation.
%
%Since the ODE is formed mainly due to the explicit structure of the feasible region of MCPP and the ``quasi"-multilinear polynomial form of the objective function, to apply to more problems that these properties may not hold still (such as mixed integer cases or instances with more constraints other than multiple choice), we should develop new technique. This may include changing the state space of the random variable.
%
%There may be more methods in numerical ODE to get the equilibrium faster. In our work, we only adopted the step size control. Other technique such as Nesterov's acceleration should also be concerned in the future.

%\section*{Data Availability}
%
%The datasets generated during and/or analyzed during the current study are available from the corresponding author on reasonable request.

\section*{Acknowledgements}

This research was supported by the National Key R\&D Program of China (Nos. 2020AAA0105200, 2022YFA1005102) and the National Natural Science Foundation of China (Nos.~12325112, 12288101, 11822102).
SS is partially supported by Beijing Academy of Artificial Intelligence (BAAI).

\normalem
%\bibliography{ref_c.bib}

\end{document}